\documentclass[12pt, twoside]{article}
\usepackage{amsmath,amssymb,amsthm}
\usepackage{bm}
\usepackage[dvipdfm]{pict2e}

\usepackage[dvipdfmx]{graphicx,color}
\makeatletter

\@addtoreset{equation}{section}
\makeatother
\theoremstyle{plain} 
\newtheorem{theorem}{Theorem}
\newtheorem*{theorem*}{Theorem}
\newtheorem{prop}[theorem]{Proposition}
\newtheorem*{prop*}{Proposition}
\newtheorem{lemma}[theorem]{Lemma}
\newtheorem*{lemma*}{Lemma}
\newtheorem{cor}[theorem]{Corollary}
\newtheorem*{cor*}{Corollary}

\newtheorem*{example*}{Example}

\newtheorem*{axiom*}{Axiom}

\newtheorem*{problem*}{Problem}

\newtheorem*{summary*}{Summary}

\newtheorem*{guide*}{Guide}
\theoremstyle{definition} 
\newtheorem{definition}[theorem]{Definition}
\newtheorem*{definition*}{Definition}
\theoremstyle{definition} 
\newtheorem{remark}[theorem]{Remark}
\newtheorem*{remark*}{Remark}
\numberwithin{theorem}{section}
\numberwithin{equation}{section}
\numberwithin{figure}{section}
\numberwithin{table}{section}

\makeatletter
\renewenvironment{proof}[1][\proofname]{\par
  \normalfont
  \topsep6\p@\@plus6\p@ \trivlist
  \item[\hskip\labelsep{\bfseries #1}\@addpunct{\bfseries.}]\ignorespaces
}{%
  \endtrivlist
}
\renewcommand{\proofname}{Proof}
\makeatother
\makeatletter
\def\BOXSYMBOL{\RIfM@\bgroup\else$\bgroup\aftergroup$\fi
  \vcenter{\hrule\hbox{\vrule height.85em\kern.6em\vrule}\hrule}\egroup}
\makeatother
\newcommand{\BOX}{%
  \ifmmode\else\leavevmode\unskip\penalty9999\hbox{}\nobreak\hfill\fi
  \quad\hbox{\BOXSYMBOL}}

\begin{document}
\title{\textbf{Elliptic Ding-Iohara Algebra and \\ the Free Field Realization of \\ the Elliptic Macdonald Operator}}
\author{\textbf{Yosuke Saito} \\ Mathematical Institute, Tohoku University, \\ Sendai, Japan.}
\maketitle
\begin{abstract}
The Ding-Iohara algebra [15][16][17] is a quantum algebra arising from the free field realization of the Macdonald operator. Starting from the elliptic kernel function introduced by Komori, Noumi and Shiraishi [18], we can define an elliptic analog of the Ding-Iohara algebra. The free field realization of the elliptic Macdonald operator is also constructed.
\end{abstract}
\tableofcontents
\pagestyle{headings}

\vskip 0.5cm
\textbf{Notations.} In this paper, we use the following symbols.
\begin{align*}
&\mathbf{Z} : \text{The set of integers}, \quad \mathbf{Z}_{\geq{0}}:=\{0,1,2,\cdots\}, \quad \mathbf{Z}_{>0}:=\{1, \, 2, \, \cdots\}, \\
&\mathbf{Q} : \text{The set of rational numbers},  \quad \mathbf{Q}(q,t) : \text{The field of rational functions of $q,\,t$ over $\mathbf{Q}$}, \\
&\mathbf{C} : \text{The set of complex numbers}, \quad \mathbf{C}^{\times}:=\mathbf{C}\setminus\{0\}, \\
&\mathbf{C}[[z,z^{-1}]] : \text{The set of formal power series of $z, \, z^{-1}$ over $\mathbf{C}$}, \\
&\text{The $q$-infinite product} : (x;q)_{\infty}:{=}\prod_{n\geq{0}}(1-xq^{n}) \,\, (|q|<1), \quad (x;q)_{n}:{=}\frac{(x;q)_{\infty}}{(q^{n}x;q)_{\infty}} \,\, (n\in{\mathbf{Z}}), \\
&\text{The theta function} : \Theta_{p}(x):=(p;p)_{\infty}(x;p)_{\infty}(px^{-1};p)_{\infty}, \\
&\text{The double infinite product} : (x;q,p)_{\infty}:=\prod_{m,n\geq{0}}(1-xq^{m}p^{n}), \\
&\text{The elliptic gamma function} : \Gamma_{q,p}(x):=\frac{(qpx^{-1};q,p)_{\infty}}{(x;q,p)_{\infty}}.
\end{align*} 
For the theta function and the elliptic gamma function, the following relations hold.
\begin{align*}
&\Theta_{p}(x)=-x\Theta_{p}(x^{-1}), \quad \Theta_{p}(px)=-x^{-1}\Theta_{p}(x), \\
&\Gamma_{q,p}(qx)=\frac{\Theta_{p}(x)}{(p;p)_{\infty}}\Gamma_{q,p}(x), \quad \Gamma_{q,p}(px)=\frac{\Theta_{q}(x)}{(q;q)_{\infty}}\Gamma_{q,p}(x).
\end{align*}

\section{Introduction}
The aims of this paper are to introduce an elliptic analog of the Ding-Iohara algebra and to construct the free field realization of the elliptic Macdonald operator. We accomplish them by starting from the elliptic kernel function defined below. Let us explain backgrounds and some motivations. 

The relation between quantum algebras and the Macdonald symmetric functions has been studied by several authors. In these works, one of the most remarkable work is the construction of the $q$-Virasoro algebra and the $q$-$W_{N}$ algebra by Awata, Odake, Kubo, and Shiraishi [10][11][12]. It is known that the singular vectors of the Virasoro algebra and the $W_{N}$ algebra correspond to the Jack symmetric functions [8]. On the other hand, the Macdonald symmetric functions are $q$-analog of the Jack symmetric functions [2][3]. Then Awata, Odake, Kubo, and Shiraishi constructed the $q$-Virasoro algebra and the $q$-$W_{N}$ algebra whose singular vectors correspond to the Macdonald symmetric functions :
\begin{align*}
\begin{matrix}
\text{$q$-Virasoro algebra, $q$-$W_{N}$ algebra} & \xrightarrow[]{\text{singular vector}}  &\text{Macdonald symmetric functions} \\
\uparrow \text{\footnotesize{$q$-deformation}} & {} &\uparrow \text{\footnotesize{$q$-deformation}} \\
\text{Virasoro algebra, $W_{N}$ algebra} & \xrightarrow[]{\text{singular vector}} &\text{Jack symmetric functions}
\end{matrix}
\end{align*}

In the middle of 2000's, another stream occurs from the free field realization of the Macdonald operator. The Macdonald operator $H_{N}(q,t)$ $(N\in\mathbf{Z}_{>0})$ is defined by
\begin{align}
H_{N}(q,t):=\sum_{i=1}^{N}\prod_{j\neq{i}}\frac{tx_{i}-x_{j}}{x_{i}-x_{j}}T_{q,x_{i}} \quad (T_{q,x_{i}}f(x_{1},\cdots, x_{N}):=f(x_{1},\cdots,qx_{i},\cdots,x_{N}))
\end{align}
and the free field realization of the Macdonald operator tells us that we can reproduce the operator by boson operators. As we will see in section 2, the free field realization of the Macdonald operator is based on the form of the kernel function defined by 
\begin{align}
\Pi(q,t)(x,y):=\prod_{i,j}\frac{(tx_{i}y_{j};q)_{\infty}}{(x_{i}y_{j};q)_{\infty}}.
\end{align}
It has been realized that from the free field realization of the Macdonald operator, a certain quantum algebra arises, \textit{the Ding-Iohara algebra} [15][16][17]. Recently this algebra has been applied to several  materials of mathematical physics, such as  the AGT conjecture [19][20][21], as well as the refined topological vertex which is used to calculate amplitudes and partition functions in the topological string theory [22].

On the other hand, in the elliptic theory side it is well-known that the Macdonald operator allows the elliptic analog defined by [1],
\begin{align}
H_{N}(q,t,p):=\sum_{i=1}^{N}\prod_{j\neq{i}}\frac{\Theta_{p}(tx_{i}/x_{j})}{\Theta_{p}(x_{i}/x_{j})}T_{q,x_{i}},
\end{align}
as well as the kernel function for this operator introduced by Komori, Noumi, and Shiraishi [18] :
\begin{align}
\Pi(q,t,p)(x,y):=\prod_{i,j}\frac{\Gamma_{q,p}(x_{i}y_{j})}{\Gamma_{q,p}(tx_{i}y_{j})}.
\end{align}
Since the free field realization of the Macdonald operator is available, that the above operator (1.3) can be derived from the free field realization will be a natural expect. In [17], Feigin, Hashizume, Hoshino, Shiraishi, Yanagida constructed the free field realization of the elliptic Macdonald operator and an elliptic analog of the Ding-Iohara algebra based on the idea of the quasi-Hopf twist. It is crucial that the authors of [17] noticed if one want to treat the elliptic Macdonald operator in the context of the free field realization, the Ding-Iohara algebra should become elliptic. However it is not clear whether the materials treated in [17] have connections to the elliptic kernel function. Hence the following problem remained open :
\begin{quote}
Construct the free field realization of the elliptic Macdonald operator $H_{N}(q,t,p)$ and the elliptic Ding-Iohara algebra which have connections to the elliptic kernel function $\Pi(q,t,p)(x,y)$.
\end{quote}

Our strategy to solve the above problem is the following. Since the free field realization of the Macdonald operator is based on the form of the kernel function, it is plausible that one can construct the free field realization of the elliptic Macdonald operator from the elliptic kernel function. It turns out this leads to another elliptic analog of the Ding-Iohara algebra :
\begin{align*}
\begin{matrix}
\genfrac{}{}{0pt}{0}{\text{Elliptic Macdonald operator}}{H_{N}(q,t,p)} & \xrightarrow[]{\textbf{free field realization !}}& \genfrac{}{}{0pt}{0}{\textbf{Elliptic Ding-Iohara algebra}}{\bm{\mathcal{U}(q,t,p)}} \\
\bigg{\uparrow}\text{\footnotesize{elliptic deformation}} & {} & \bigg{\uparrow}\text{\footnotesize{\textbf{elliptic deformation !}}} \\
\text{Macdonald operator} \, \, H_{N}(q,t) & \xrightarrow[]{\text{\quad free field realization \quad}}& \text{Ding-Iohara algebra} \,\, \mathcal{U}(q,t)
\end{matrix}
\end{align*}

Our main results are as follows.

\medskip
\begin{definition}[Definition 3.9 in section 3) (\textbf{Elliptic Ding-Iohara algebra $\mathcal{U}(q,t,p)$}] 
Let us define the structure function $g_{p}(x)$ by
\begin{align*}
g_{p}(x):=\frac{\Theta_{p}(qx)\Theta_{p}(t^{-1}x)\Theta_{p}(q^{-1}tx)}{\Theta_{p}(q^{-1}x)\Theta_{p}(tx)\Theta_{p}(qt^{-1}x)}.
\end{align*}
Here we have used the notation in page 2, and assume $|q|<1$, $|p|<1$. We define the elliptic Ding-Iohara algebra $\mathcal{U}(q,t,p)$ to be the associative $\mathbf{C}$-algebra generated by $\{x^{\pm}_{n}(p)\}_{n\in{\mathbf{Z}}}$, $\{\psi^{\pm}_{n}(p)\}_{n\in{\mathbf{Z}}}$ and $\gamma$ subject to the following relation : we set $\gamma$ as the central, invertible element and currents to be $x^{\pm}(p;z):=\sum_{n\in{\mathbf{Z}}}x^{\pm}_{n}(p)z^{-n}$, $\psi^{\pm}(p;z):=\sum_{n\in{\mathbf{Z}}}\psi^{\pm}_{n}(p)z^{-n}$.
\begin{align*}
&\hskip 1cm [\psi^{\pm}(p;z), \psi^{\pm}(p;w)]=0, \quad \psi^{+}(p;z)\psi^{-}(p;w)=\frac{g_{p}(\gamma z/w)}{g_{p}(\gamma^{-1}z/w)}\psi^{-}(p;w)\psi^{+}(p;z),\notag\\
&\hskip 3cm \psi^{\pm}(p;z)x^{+}(p;w)=g_{p}\Big(\gamma^{\pm\frac{1}{2}}\frac{z}{w}\Big)x^{+}(p;w)\psi^{\pm}(p;z),\notag\\
&\hskip 3cm \psi^{\pm}(p;z)x^{-}(p;w)=g_{p}\Big(\gamma^{\mp\frac{1}{2}}\frac{z}{w}\Big)^{-1}x^{-}(p;w)\psi^{\pm}(p;z),\notag\\
&\hskip 3cm x^{\pm}(p;z)x^{\pm}(p;w)=g_{p}\Big(\frac{z}{w}\Big)^{\pm 1}x^{\pm}(p;w)x^{\pm}(p;z),\notag\\
&[x^{+}(p;z),x^{-}(p;w)]
=\frac{\Theta_{p}(q)\Theta_{p}(t^{-1})}{(p;p)_{\infty}^{3}\Theta_{p}(qt^{-1})}\bigg\{\delta\Big(\gamma\frac{w}{z}\Big)\psi^{+}(p;\gamma^{1/2}w)-\delta\Big(\gamma^{-1}\frac{w}{z}\Big)\psi^{-}(p;\gamma^{-1/2}w)\bigg\},
\end{align*}
where we set the delta function $\delta(z):=\sum_{n\in{\mathbf{Z}}}z^{n}$.
\end{definition}

The free field realization of the elliptic Ding-Iohara algebra $\mathcal{U}(q,t,p)$ is constructed as follows. First for the theta function $\Theta_{p}(x)$ one can check the following :
\begin{align*}
\Theta_{p}(x) \xrightarrow[p \to 0]{} 1-x.
\end{align*}
On the other hand, we can rewrite $1-x$ and $\Theta_{p}(x)$ as follows :
\begin{align*}
1-x &=\exp\Big(\log(1-x)\Big)=\exp\bigg(-\sum_{n>0}\frac{x^{n}}{n}\bigg) \quad (|x|<1), \\
\Theta_{p}(x) &=(p;p)_{\infty}(x;p)_{\infty}(px^{-1};p)_{\infty} \\
&=(p;p)_{\infty}\exp\Big(\log(x;p)_{\infty}(px^{-1};p)_{\infty}\Big) \\
&=(p;p)_{\infty}\exp\bigg(-\sum_{n>0}\frac{p^{n}}{1-p^{n}}\frac{x^{-n}}{n}\bigg)\exp\bigg(-\sum_{n>0}\frac{1}{1-p^{n}}\frac{x^{n}}{n}\bigg) \quad (|p|<|x|<1).
\end{align*}
From these expressions, one can notice a procedure of the elliptic deformation as follows :
\begin{align*}
1-x &=\exp\bigg(-\sum_{n>0}\frac{x^{n}}{n}\bigg) 
\xrightarrow[\genfrac{}{}{0pt}{1}{\text{elliptic}}{\text{deformation}}]{} \exp\bigg(-\sum_{n>0}\frac{p^{n}}{1-p^{n}}\frac{x^{-n}}{n}\bigg)\exp\bigg(-\sum_{n>0}\frac{1}{1-p^{n}}\frac{x^{n}}{n}\bigg) \\
&=\frac{\Theta_{p}(x)}{(p;p)_{\infty}}.
\end{align*}
We can also recognize the above process as follows :
\begin{align*}
&(1) \, \text{Take the substitution as} \quad 1-x=\exp\bigg(-\sum_{n>0}\frac{x^{n}}{n}\bigg) \to \exp\bigg(-\sum_{n>0}\frac{1}{1-p^{n}}\frac{x^{n}}{n}\bigg). \\
&(2) \, \text{Multiply the above by the negative power part of $x$ as} \\
&\exp\bigg(-\sum_{n>0}\frac{1}{1-p^{n}}\frac{x^{n}}{n}\bigg) \to \exp\bigg(-\sum_{n>0}\frac{p^{n}}{1-p^{n}}\frac{x^{-n}}{n}\bigg)\exp\bigg(-\sum_{n>0}\frac{1}{1-p^{n}}\frac{x^{n}}{n}\bigg)=\frac{\Theta_{p}(x)}{(p;p)_{\infty}}.
\end{align*}
As is shown in this paper, for boson operators the procedure of the elliptic deformation similar to the above process is available (for example, see the proposition 3.1 in section 3 or the definition 5.1 in section 5). Using two sets of boson generators, we can reproduce the theta function and the elliptic gamma function from OPE (Operator Product Expansion) of boson operators. Consequently, we have the following.

\begin{theorem}[Theorem 3.8 in section 3) \, (\textbf{Free field realization of the elliptic Ding-Iohara algebra $\mathcal{U}(q,t,p)$}] 
Let us define an algebra $\mathcal{B}_{a,b}$ of bosons : it is generated by $a=\{a_{n}\}_{n\in{\mathbf{Z}\setminus\{0\}}}$, $b=\{b_{n}\}_{n\in{\mathbf{Z}\setminus\{0\}}}$ with the following relations :
\begin{align*}
&[a_{m},a_{n}]=m(1-p^{|m|})\frac{1-q^{|m|}}{1-t^{|m|}}\delta_{m+n,0}, \quad [b_{m},b_{n}]=m\frac{1-p^{|m|}}{(qt^{-1}p)^{|m|}}\frac{1-q^{|m|}}{1-t^{|m|}}\delta_{m+n,0}, \\
&[a_{m},b_{n}]=0.
\end{align*}
We define the boson Fock space $\mathcal{F}$ to be the left $\mathcal{B}_{a,b}$ module generated by the vacuum vector $|0 \rangle$ which satisfies $a_{n}|0 \rangle=b_{n}|0 \rangle=0 \, (n>0)$ :
\begin{align*}
\mathcal{F}=\text{\rm span}\{a_{-\lambda}b_{-\mu}|0 \rangle : \lambda, \mu\in{\mathcal{P}}\},
\end{align*}
where $\mathcal{P}$ denotes the set of partitions and $a_{-\lambda}:=a_{-\lambda_{1}}\cdots a_{-\lambda_{\ell(\lambda)}}$ for a partition $\lambda$. Set $\gamma:=(qt^{-1})^{-1/2}$ and define operators $\eta(p;z),\,\xi(p;z),\,\varphi^{\pm}(p;z) : \mathcal{F} \to \mathcal{F}\otimes\mathbf{C}[[z,z^{-1}]]$ as follows :
\begin{align*}
&\eta(p;z):=\bm{:}\exp\bigg(-\sum_{n\neq{0}}\frac{1-t^{-n}}{1-p^{|n|}}p^{|n|}b_{n}\frac{z^{n}}{n}\bigg)\exp\bigg(-\sum_{n\neq{0}}\frac{1-t^{n}}{1-p^{|n|}}a_{n}\frac{z^{-n}}{n}\bigg)\bm{:}, \\
&\xi(p;z)
:=\bm{:}\exp\bigg(\sum_{n\neq{0}}\frac{1-t^{-n}}{1-p^{|n|}}\gamma^{-|n|}p^{|n|}b_{n}\frac{z^{n}}{n}\bigg)\exp\bigg(\sum_{n\neq{0}}\frac{1-t^{n}}{1-p^{|n|}}\gamma^{|n|}a_{n}\frac{z^{-n}}{n}\bigg)\bm{:}, \\
&\varphi^{+}(p;z):=\bm{:}\eta(p;\gamma^{1/2}z)\xi(p;\gamma^{-1/2}z)\bm{:}, \quad \varphi^{-}(p;z):=\bm{:}\eta(p;\gamma^{-1/2}z)\xi(p;\gamma^{1/2}z)\bm{:}.
\end{align*}
Then the map defined by
\begin{align*}
x^{+}(p;z) \mapsto \eta(p;z), \quad x^{-}(p;z) \mapsto \xi(p;z), \quad \psi^{\pm}(p;z) \mapsto \varphi^{\pm}(p;z)
\end{align*}
gives a representation of the elliptic Ding-Iohara algebra $\mathcal{U}(q,t,p)$.
\end{theorem}

\begin{theorem}[Theorem 4.4 in section 4) (\textbf{Free field realization of the elliptic Macdonald operator}]
Let $\phi(p;z) : \mathcal{F} \to \mathcal{F}\otimes\mathbf{C}[[z,z^{-1}]]$ be an operator defined as 
\begin{align*}
\phi(p;z):=\exp\bigg(\sum_{n>0}\frac{(1-t^{n})(qt^{-1}p)^{n}}{(1-q^{n})(1-p^{n})}b_{-n}\frac{z^{-n}}{n}\bigg)\exp\bigg(\sum_{n>0}\frac{1-t^{n}}{(1-q^{n})(1-p^{n})}a_{-n}\frac{z^{n}}{n}\bigg).
\end{align*}
We use the notation $\phi_{N}(p;x):=\prod_{j=1}^{N}\phi(p;x_{j})$ $(N\in\mathbf{Z}_{>0})$. Then the operator $\eta(p;z)$ in the Theorem 1.2 and the operator $\phi(p;z)$ reproduce the elliptic Macdonald operator $H_{N}(q,t,p)$ as follows :
\begin{align*}
[\eta(p;z)-t^{-N}(\eta(p;z))_{-}(\eta(p;p^{-1}z))_{+}]_{1}\phi_{N}(p;x)|0 \rangle 
=\frac{t^{-N+1}\Theta_{p}(t^{-1})}{(p;p)_{\infty}^{3}}H_{N}(q,t,p)\phi_{N}(p;x)|0 \rangle,
\end{align*}
where $(\eta(p;z))_{\pm}$ stands for the plus and minus parts of $\eta(p;z)$ respectively as
\begin{align*}
(\eta(p;z))_{\pm}=\exp\bigg(-\sum_{\pm n>0}\frac{1-t^{-n}}{1-p^{|n|}}p^{|n|}b_{n}\frac{z^{n}}{n}\bigg)\exp\bigg(-\sum_{\pm n>0}\frac{1-t^{n}}{1-p^{|n|}}a_{n}\frac{z^{-n}}{n}\bigg)
\end{align*}
and $[f(z)]_{1}$ denotes the constant term of $f(z)$ in $z$.
\end{theorem}

\textbf{Organization of this paper.}

This paper is organized as follows. In section 2, we give a review of the trigonometric case. In section 3, we show how we can obtain the elliptic Ding-Iohara algebra. First, we define the elliptic kernel function introduced by Komori, Noumi and Shiraishi [18]. This is regarded as an important function to construct an elliptic analog of the Macdonald symmetric functions. Second, from the elliptic kernel function, we define elliptic currents denoted by $\eta(p;z),\, \xi(p;z),$ and $\varphi^{\pm}(p;z)$ which satisfy elliptic deformed relations of Ding-Iohara algebra's. Consequently, we can define the elliptic Ding-Iohara algebra $\mathcal{U}(q,t,p).$

In section 4, to clarify whether the elliptic Macdonald operator can be represented by $\eta(p;z)$, we study relations between the elliptic current $\eta(p;z)$ and the elliptic Macdonald operator. We derive the free field realization for the elliptic Macdonald operator in the form of the Theorem 1.3.

In section 5, some observations and remarks are given, and section 6 is appendix which contains proofs of Wick's theorem, Ramanujan's summation formula, the formula of the partial fraction expansion involving the theta functions.

\section{A review of the trigonometric case}
In this section, before considering the elliptic case we review some materials which construct backgrounds of this paper ; the Macdonald symmetric functions, the free field realization of the Macdonald operator, and the Ding-Iohara algebra.

\subsection{Macdonald symmetric functions}
First, we give some notations of symmetric polynomials and symmetric functions [2][3][14]. Let $q,t\in{\mathbf{C}}$ be parameters and assume $|q|<1$. We denote the $N$-th symmetric group by $\mathfrak{S}_{N}$ and set $\Lambda_{N}(q,t):=\mathbf{Q}(q,t)[x_{1},\cdots,x_{N}]^{\mathfrak{S}_{N}}$ as the space of $N$-variables symmetric polynomials over $\mathbf{Q}(q,t)$. If  a sequence $\lambda=(\lambda_{1}, \cdots, \lambda_{N})\in{(\mathbf{Z}_{\geq{0}})^{N}}$ satisfies the condition $\lambda_{i}\geq{\lambda_{i+1}} \,(1\leq{\forall i}\leq{N})$, $\lambda$ is called a partition. We denote the set of partitions by $\mathcal{P}$. For a partition $\lambda$, $\ell(\lambda):=\sharp\{i : \lambda_{i}\neq{0}\}$ denotes the length of $\lambda$, and $|\lambda|:=\sum_{i=1}^{\ell(\lambda)}\lambda_{i}$ denotes the size of $\lambda$.

For $\alpha=(\alpha_{1}, \cdots, \alpha_{N})\in{(\mathbf{Z}_{\geq{0}})^{N}}$, we set $x^{\alpha}:=x_{1}^{\alpha_{1}}\cdots x_{N}^{\alpha_{N}}$. For a partition $\lambda$, we define the monomial symmetric polynomial $m_{\lambda}(x)$ as follows :
\begin{align*}
m_{\lambda}(x):=\sum_{\text{$\alpha$ : $\alpha$ is a permutation of $\lambda$}}x^{\alpha}.
\end{align*}
As is well-known, $\{m_{\lambda}(x)\}_{\lambda\in{\mathcal{P}}}$ form a basis of $\Lambda_{N}(q,t)$. Let $p_{n}(x):=\sum_{i=1}^{N}x_{i}^{n} \,(n\in{\mathbf{Z}_{>0}})$ be the power sum and for a partition $\lambda$, we define $p_{\lambda}(x):=p_{\lambda_{1}}(x)\cdots p_{\lambda_{\ell(\lambda)}}(x)$.

Let $\rho^{N+1}_{N} : \Lambda_{N+1}(q,t) \to \Lambda_{N}(q,t)$ be the homomorphism defined by 
\begin{align*}
(\rho^{N+1}_{N}f)(x_{1},\cdots,x_{N}):=f(x_{1},\cdots,x_{N},0) \quad (f\in\Lambda_{N+1}(q,t)).
\end{align*}
Let us define the ring of symmetric functions $\Lambda(q,t)$ as the projective limit defined by $\{\rho^{N+1}_{N}\}_{N\geq{1}}$ :
\begin{align*}
\Lambda(q,t):=\lim_{\longleftarrow}\Lambda_{N}(q,t). 
\end{align*}
It is known that $\{p_{\lambda}(x)\}_{\lambda\in{\mathcal{P}}}$ form a basis of $\Lambda(q,t)$. We define $n_{\lambda}(a):=\sharp\{i : \lambda_{i}=a\}$ and $z_{\lambda},\,z_{\lambda}(q,t)$ as
\begin{align*}
z_{\lambda}:=\prod_{a\geq{1}}a^{n_{\lambda}(a)}n_{\lambda}(a)!, \quad z_{\lambda}(q,t):=z_{\lambda}\prod_{i=1}^{\ell(\lambda)}\frac{1-q^{\lambda_{i}}}{1-t^{\lambda_{i}}}. 
\end{align*}
Then we define an inner product $\langle \,, \rangle_{q,t}$ as follows :
\begin{align}
\langle p_{\lambda}(x), p_{\mu}(x) \rangle_{q,t}=\delta_{\lambda\mu}z_{\lambda}(q,t).
\end{align}
We define the kernel function $\Pi(q,t)(x,y)$ as follows :
\begin{align}
\Pi(q,t)(x,y):=\prod_{i,j}\frac{(tx_{i}y_{j};q)_{\infty}}{(x_{i}y_{j};q)_{\infty}}.
\end{align}
Then we have the following.
\begin{align*}
\sum_{\lambda\in{\mathcal{P}}}\frac{1}{z_{\lambda}(q,t)}p_{\lambda}(x)p_{\lambda}(y)=\Pi(q,t)(x,y).
\end{align*}

\begin{remark}
Assume that $u_{\lambda}(x),\,v_{\lambda}(x)$ $(\lambda\in\mathcal{P})$ are homogeneous symmetric functions whose degree are $|\lambda|$, and $\{u_{\lambda}(x)\}_{\lambda\in{\mathcal{P}}}$ and $\{v_{\lambda}(x)\}_{\lambda\in{\mathcal{P}}}$ form a basis of $\Lambda(q,t)$ respectively. Then the following holds :
\begin{align*}
&\hskip 1cm \text{$\{u_{\lambda}(x)\}_{\lambda\in{\mathcal{P}}}$ and $\{v_{\lambda}(x)\}_{\lambda\in{\mathcal{P}}}$ are dual basis under the inner product $\langle\, , \rangle_{q,t}$}. \\
&\Longleftrightarrow \text{$\{u_{\lambda}(x)\}_{\lambda\in{\mathcal{P}}}$ and $\{v_{\lambda}(x)\}_{\lambda\in{\mathcal{P}}}$ satisfy the relation} \, \sum_{\lambda\in{\mathcal{P}}}u_{\lambda}(x)v_{\lambda}(y)=\Pi(q,t)(x,y).
\end{align*}
Due to this fact, the form of the inner product $\langle\,,\rangle_{q,t}$ is determined by the form of the kernel function $\Pi(q,t)(x,y)$.
\end{remark}

The Macdonald symmetric functions are $q$-analog of the Schur symmetric functions and the Jack symmetric functions. The existence of the Macdonald symmetric functions due to Macdonald is stated as follows [2][3][14]. We define the order in $\mathcal{P}$ as follows :
\begin{align*}
\lambda\geq{\mu} \Longleftrightarrow \text{$|\lambda|=|\mu|$ and for all $i$, $\lambda_{1}+\cdots+\lambda_{i}\geq{\mu_{1}+\cdots+\mu_{i}}$}.
\end{align*}

\begin{theorem}[\textbf{Existence theorem of the Macdonald symmetric functions}] 
For each partition $\lambda$, a symmetric function $P_{\lambda}(x)\in{\Lambda(q,t)}$ satisfying the following conditions uniquely exists.
\begin{align}
&(1) \quad P_{\lambda}(x)=\sum_{\mu\leq{\lambda}}u_{\lambda\mu}m_{\mu}(x) \quad (u_{\lambda\mu}\in{\mathbf{Q}(q,t)}), \\
&(2) \quad \lambda \neq \mu \Longrightarrow \langle P_{\lambda}(x), P_{\mu}(x)\rangle_{q,t}=0.
\end{align}
\end{theorem}

\begin{remark}
Set $\langle \lambda \rangle_{q,t}:=\langle P_{\lambda}(x),P_{\lambda}(x) \rangle_{q,t}$. Then the Macdonald symmetric functions satisfy the following relation.
\begin{align*}
\sum_{\lambda\in{\mathcal{P}}}\frac{1}{\langle \lambda \rangle_{q,t}}P_{\lambda}(x)P_{\lambda}(y)=\Pi(q,t)(x,y).
\end{align*}
This means that $\{P_{\lambda}(x)\}_{\lambda\in{\mathcal{P}}}$ form a basis of $\Lambda(q,t)$.
\end{remark}

For the Macdonald symmetric function $P_{\lambda}(x)$, we define the $N$-variable symmetric polynomial $P_{\lambda}(x_{1},\cdots,x_{N})$ as $P_{\lambda}(x_{1},\cdots,x_{N}):=P_{\lambda}(x_{1},\cdots,x_{N},0,0,\cdots)$. We call it the $N$-variables Macdonald polynomials. We set the $q$-shift operator by 
\begin{align*}
T_{q,x_{i}}f(x_{1},\cdots, x_{N}):=f(x_{1},\cdots,qx_{i},\cdots,x_{N})
\end{align*}
and define the Macdonald operator $H_{N}(q,t) : \Lambda_{N}(q,t) \to \Lambda_{N}(q,t)$ as follows :
\begin{align}
H_{N}(q,t):=\sum_{i=1}^{N}\prod_{j\neq{i}}\frac{tx_{i}-x_{j}}{x_{i}-x_{j}}T_{q,x_{i}}.
\end{align}

\begin{prop}
(1) For each partition $\lambda \, (\ell(\lambda)\leq{N})$, $P_{\lambda}(x_{1},\cdots,x_{N})$ is an eigen function of the Macdonald operator :
\begin{align}
H_{N}(q,t)P_{\lambda}(x_{1},\cdots,x_{N})=\varepsilon_{N}(\lambda)P_{\lambda}(x_{1},\cdots,x_{N}), \quad \varepsilon_{N}(\lambda):=\sum_{i=1}^{N}q^{\lambda_{i}}t^{N-i}.
\end{align}
(2) The kernel function $\Pi(q,t)(x,y)$ and the Macdonald operator $H_{N}(q,t)$ satisfy
\begin{align}
H_{N}(q,t)_{x}\Pi(q,t)(x_{1},\cdots,x_{N},y_{1},\cdots,y_{N})=H_{N}(q,t)_{y}\Pi(q,t)(x_{1},\cdots,x_{N},y_{1},\cdots,y_{N}).
\end{align}
Here we set $H_{N}(q,t)_{x},\,H_{N}(q,t)_{y}$ as
\begin{align*}
H_{N}(q,t)_{x}:=\sum_{i=1}^{N}\prod_{j\neq{i}}\frac{tx_{i}-x_{j}}{x_{i}-x_{j}}T_{q,x_{i}}, \quad H_{N}(q,t)_{y}:=\sum_{i=1}^{N}\prod_{j\neq{i}}\frac{ty_{i}-y_{j}}{y_{i}-y_{j}}T_{q,y_{i}}.
\end{align*}
\end{prop}

\subsection{Free field realization of the Macdonald operator}
In this subsection, we show the free field realization of the Macdonald operator [14]. In the following, let $q,t\in{\mathbf{C}}$ be parameters and we assume $|q|<1$. First we define the algebra $\mathcal{B}$ of boson to be generated by $\{a_{n}\}_{n\in{\mathbf{Z}\setminus\{0\}}}$ and the relation :
\begin{align}
[a_{m},a_{n}]=m\frac{1-q^{|m|}}{1-t^{|m|}}\delta_{m+n,0}.
\end{align}
We set the normal ordering $\bm{:} \bullet \bm{:}$ as 
\begin{align*}
\bm{:}a_{m}a_{n}\bm{:}=
\begin{cases}
a_{m}a_{n} \quad (m<n), \\
a_{n}a_{m} \quad (m\geq{n}).
\end{cases}
\end{align*}
Let $|0 \rangle$ be the vacuum vector which satisfies $a_{n}|0 \rangle=0 \, (n>0)$. For a partition $\lambda$, we set $a_{-\lambda}:=a_{-\lambda_{1}}\cdots a_{-\lambda_{\ell(\lambda)}}$ and define the boson Fock space $\mathcal{F}$ as a left $\mathcal{B}$ module :
\begin{align*}
\mathcal{F}:=\text{span}\{a_{-\lambda}|0 \rangle : \lambda\in{\mathcal{P}}\}.
\end{align*}
We set the dual vacuum vector $\langle 0|$ which satisfies $\langle 0|a_{n}=0 \, (n<0)$. Similar to the definition of $\mathcal{F}$, we define the dual boson Fock space $\mathcal{F}^{\ast}$, as a right $\mathcal{B}$ module :
\begin{align*}
\mathcal{F}^{\ast}:=\text{span}\{\langle 0|a_{\lambda} : \lambda\in{\mathcal{P}}\} \quad (a_{\lambda}:=a_{\lambda_{1}}\cdots a_{\lambda_{\ell(\lambda)}}).
\end{align*}
Let us define a bilinear form $\langle \bullet|\bullet \rangle : \mathcal{F}^{\ast} \times \mathcal{F} \to \mathbf{C}$ by the following conditions.
\begin{align*}
(1) \, \langle 0|0 \rangle=1, \quad (2) \, \langle 0|a_{\lambda}a_{-\mu}|0 \rangle=\delta_{\lambda\mu}z_{\lambda}(q,t).
\end{align*}

\begin{remark}
It is clear that the bilinear form defined above corresponds to the inner product $\langle \,,\rangle_{q,t}$ in (2.1). Therefore the relation (2.8) is determined by the inner product $\langle\,,\rangle_{q,t}$, or equivalently, by the form of the kernel function $\Pi(q,t)(x,y)$.
\end{remark}

To reproduce the Macdonald operator from a boson operator, let us define an operator $\eta(z), \xi(z) : \mathcal{F} \to \mathcal{F}\otimes \mathbf{C}[[z,z^{-1}]]$ as follows $(\gamma:=(qt^{-1})^{-1/2})$.
\begin{align}
\eta(z):=\bm{:}\exp\bigg(-\sum_{n\neq{0}}(1-t^{n})a_{n}\frac{z^{-n}}{n}\bigg)\bm{:}, \quad
\xi(z):=\bm{:}\exp\bigg(\sum_{n\neq{0}}(1-t^{n})\gamma^{|n|}a_{n}\frac{z^{-n}}{n}\bigg)\bm{:}.
\end{align}
We can check that $\eta(z)$, $\xi(z)$ satisfy the relation :
\begin{align}
&\eta(z)\eta(w)=\frac{(1-w/z)(1-qt^{-1}w/z)}{(1-qw/z)(1-t^{-1}w/z)}\bm{:}\eta(z)\eta(w)\bm{:}, \\
&\xi(z)\xi(w)=\frac{(1-w/z)(1-q^{-1}tw/z)}{(1-q^{-1}w/z)(1-tw/z)}\bm{:}\xi(z)\xi(w)\bm{:}.
\end{align}
Define operators $\phi(z) : \mathcal{F} \to \mathcal{F}\otimes\mathbf{C}[[z,z^{-1}]]$, $\phi^{\ast}(z) : \mathcal{F}^{\ast} \to \mathcal{F}^{\ast}\otimes\mathbf{C}[[z,z^{-1}]]$ as follows :
\begin{align}
\phi(z):=\exp\bigg(\sum_{n>0}\frac{1-t^{n}}{1-q^{n}}a_{-n}\frac{z^{n}}{n}\bigg), \quad \phi^{\ast}(z):=\exp\bigg(\sum_{n>0}\frac{1-t^{n}}{1-q^{n}}a_{n}\frac{z^{n}}{n}\bigg).
\end{align}
Then we can check the relations.
\begin{align*}
& \eta(z)\phi(w)=\frac{1-w/z}{1-tw/z}\bm{:}\eta(z)\phi(w)\bm{:}, \quad \bm{:}\eta(tz)\phi(z)\bm{:}|0 \rangle=\phi(qz)|0 \rangle, \\
& \xi(z)\phi(w)=\frac{1-t\gamma w/z}{1-\gamma w/z}\bm{:}\xi(z)\phi(w)\bm{:}, \quad \bm{:}\xi(\gamma z)\phi(z)\bm{:}|0 \rangle=\phi(q^{-1}z)|0 \rangle.
\end{align*}
They are shown in the following way. By Wick's theorem we have 
\begin{align*}
\eta(z)\phi(w)
&=\exp\bigg(-\sum_{m>0}(1-t^{m})\frac{1-t^{m}}{1-q^{m}}\cdot m\frac{1-q^{m}}{1-t^{m}}\frac{(w/z)^{m}}{m\cdot m}\bigg)\bm{:}\eta(z)\phi(w)\bm{:} \\
&=\exp\bigg(-\sum_{m>0}(1-t^{m})\frac{(w/z)^{m}}{m}\bigg)\bm{:}\eta(z)\phi(w)\bm{:} \\
&=\frac{1-w/z}{1-tw/z}\bm{:}\eta(z)\phi(w)\bm{:}, \\
\xi(z)\phi(w)
&=\exp\bigg(\sum_{m>0}(1-t^{m})\gamma^{m}\frac{(w/z)^{m}}{m}\bigg)\bm{:}\xi(z)\phi(w)\bm{:} \\
&=\frac{1-t\gamma w/z}{1-\gamma w/z}\bm{:}\xi(z)\phi(w)\bm{:},
\end{align*}
where we use $\log(1-x)=-\displaystyle \sum_{n>0}\frac{x^{n}}{n} \, (|x|<1)$. The rest equations follow from simple calculations.

Set $\phi_{N}(x):=\prod_{j=1}^{N}\phi(x_{j})$ $(N\in\mathbf{Z}_{>0})$. Then we have the following.

\begin{prop}
Constant terms of $\eta(z)$, $\xi(z)$ act on $\phi_{N}(x)|0 \rangle$ as follows :
\begin{align}
& [\eta(z)]_{1}\phi_{N}(x)|0 \rangle=t^{-N}\{(t-1)H_{N}(q,t)+1\}\phi_{N}(x)|0 \rangle, \\
& [\xi(z)]_{1}\phi_{N}(x)|0 \rangle=t^{N}\{(t^{-1}-1)H_{N}(q^{-1},t^{-1})+1\}\phi_{N}(x)|0 \rangle.
\end{align}
\end{prop}

\begin{proof}[\textit{Proof}]
Here we show (2.13). From the relation of $\eta(z)$ and $\phi(z)$, we have
\begin{align*}
\eta(z)\phi_{N}(x)=\prod_{i=1}^{N}\frac{1-x_{i}/z}{1-tx_{i}/z}\bm{:}\eta(z)\phi_{N}(x)\bm{:}.
\end{align*}
By the formula of the partial fraction expansion, we have the following :
\begin{align*}
\prod_{i=1}^{N}\frac{1-x_{i}/z}{1-tx_{i}/z}=\frac{1-t}{1-t^{N}}\sum_{i=1}^{N}\frac{1-t^{-N}tx_{i}/z}{1-tx_{i}/z}\prod_{j\neq{i}}\frac{tx_{i}-x_{j}}{x_{i}-x_{j}}.
\end{align*}
Furthermore, we use the formal expression of the delta function $\delta(x)$ :
\begin{align*}
\delta(x)=\sum_{n\in{\mathbf{Z}}}x^{n}=\frac{1}{1-x}+\frac{x^{-1}}{1-x^{-1}}.
\end{align*}
This should be recognized as an expression of the Sato hyperfunction [5]. Then we have 
\begin{align*}
\prod_{i=1}^{N}\frac{1-x_{i}/z}{1-tx_{i}/z} 
&=\frac{1-t}{1-t^{N}}\sum_{i=1}^{N}(1-t^{-N}tx_{i}/z)\bigg\{\delta\Big(t\frac{x_{i}}{z}\Big)-\frac{t^{-1}x_{i}^{-1}z}{1-t^{-1}x_{i}^{-1}z}\bigg\}\prod_{j\neq{i}}\frac{tx_{i}-x_{j}}{x_{i}-x_{j}}\\
&=t^{-N}(t-1)\sum_{i=1}^{N}\prod_{j\neq{i}}\frac{tx_{i}-x_{j}}{x_{i}-x_{j}}\delta\Big(t\frac{x_{i}}{z}\Big)+t^{-N}\prod_{i=1}^{N}\frac{1-z/x_{i}}{1-t^{-1}z/x_{i}}.
\end{align*}
By this equation, the following holds :
\begin{align*}
&\quad [\eta(z)]_{1}\phi_{N}(x)|0 \rangle \\
&=\left\{t^{-N}(t-1)\sum_{i=1}^{N}\prod_{j\neq{i}}\frac{tx_{i}-x_{j}}{x_{i}-x_{j}}(\eta(tx_{i}))_{-}+t^{-N}\left[\prod_{i=1}^{N}\frac{1-z/x_{i}}{1-t^{-1}z/x_{i}}(\eta(z))_{-}\right]_{1}\right\}\phi_{N}(x)|0 \rangle \\
&=t^{-N}(t-1)\sum_{i=1}^{N}\prod_{j\neq{i}}\frac{tx_{i}-x_{j}}{x_{i}-x_{j}}T_{q,x_{i}}\phi_{N}(x)|0 \rangle +t^{-N}\phi_{N}(x)|0 \rangle \quad (\because \, (\eta(tz))_{-}\phi(z)=\phi(qz))\\
&=t^{-N}\{(t-1)H_{N}(q,t)+1\}\phi_{N}(x)|0 \rangle,
\end{align*}
where we use the equation 
\begin{align*}
\left[\prod_{i=1}^{N}\frac{1-z/x_{i}}{1-t^{-1}z/x_{i}}(\eta(z))_{-}\right]_{1}=1.
\end{align*} 
The proof of (2.14) is similar to the above, thus we omit it. \quad $\Box$
\end{proof}

\begin{remark}
(1) Set the operator $\phi^{\ast}_{N}(x):=\prod_{j=1}^{N}\phi^{\ast}(x_{j})$ $(N\in\mathbf{Z}_{>0})$. Then the kernel function $\Pi(q,t)(x,y)$ is reproduced by the operators $\phi_{N}^{\ast}(x)$, $\phi_{N}(y)$ as
\begin{align*}
\langle 0|\phi^{\ast}_{N}(x)\phi_{N}(y)|0 \rangle=\Pi(q,t)(x,y)=\prod_{1\leq{i,j}\leq{N}}\frac{(tx_{i}y_{j};q)_{\infty}}{(x_{i}y_{j};q)_{\infty}}.
\end{align*}

(2) Let us recall that the kernel function $\Pi(q,t)(x,y)$ determines the form of the relation (2.8). Therefore we can understand that the free field realization of the Macdonald operator is based on the form of the kernel function $\Pi(q,t)(x,y)$.
\end{remark}

\subsection{Ding-Iohara algebra $\mathcal{U}(q,t)$}
As is seen in the previous subsection, we can represent the Macdonald operator by using $\eta(z),\,\xi(z)$. By Wick's theorem, we can show the following.

\begin{prop}[\textbf{Relations of $\eta(z),\,\xi(z)$ and $\varphi^{\pm}(z)$}] 
Set $\gamma=(qt^{-1})^{-1/2}$ and let us define operators $\varphi^{\pm}(z) : \mathcal{F} \to \mathcal{F}\otimes\mathbf{C}[[z,z^{-1}]]$ as
\begin{align}
\varphi^{+}(z):=\bm{:}\eta(\gamma^{1/2}z)\xi(\gamma^{-1/2}z)\bm{:}, \quad \varphi^{-}(z):=\bm{:}\eta(\gamma^{-1/2}z)\xi(\gamma^{1/2}z)\bm{:}.
\end{align}
We set the structure function $g(x)$ as
\begin{align}
g(x):=\frac{(1-qx)(1-t^{-1}x)(1-q^{-1}tx)}{(1-q^{-1}x)(1-tx)(1-qt^{-1}x)}.
\end{align}
Then operators $\eta(z),\, \xi(z)$ and $\varphi^{\pm}(z)$ enjoy the following relations.
\begin{align}
&\hskip 1cm [\varphi^{\pm}(z), \varphi^{\pm}(w)]=0, \quad \varphi^{+}(z)\varphi^{-}(w)=\frac{g(\gamma z/w)}{g(\gamma^{-1}z/w)}\varphi^{-}(w)\varphi^{+}(z),\notag\\
&\varphi^{\pm}(z)\eta(w)=g\Big(\gamma^{\pm\frac{1}{2}}\frac{z}{w}\Big)\eta(w)\varphi^{\pm}(z), \quad \varphi^{\pm}(z)\xi(w)=g\Big(\gamma^{\mp\frac{1}{2}}\frac{z}{w}\Big)^{-1}\xi(w)\varphi^{\pm}(z),\notag\\
&\hskip 1cm \eta(z)\eta(w)=g\Big(\frac{z}{w}\Big)\eta(w)\eta(z),\quad \xi(z)\xi(w)=g\Big(\frac{z}{w}\Big)^{-1}\xi(w)\xi(z),\notag\\
&[\eta(z),\xi(w)]=\frac{(1-q)(1-t^{-1})}{1-qt^{-1}}\bigg\{\delta\Big(\gamma\frac{w}{z}\Big)\varphi^{+}(\gamma^{1/2}w)-\delta\Big(\gamma^{-1}\frac{w}{z}\Big)\varphi^{-}(\gamma^{-1/2}w)\bigg\}.
\end{align}
\end{prop}

\begin{remark}
(1) In section 3, we will prove an elliptic version of the proposition 2.8. Therefore we omit the proof of the proposition 2.8.

(2) As $[\varphi^{\pm}(z)]_{1}=1$, this leads that $[[\eta(z)]_{1},[\xi(w)]_{1}]=0$. The equation corresponds to the commutativity of the Macdonald operators : $[H_{N}(q,t),H_{N}(q^{-1},t^{-1})]=0$.
\end{remark}

It is important that these relations (2.17) are similar to the relations of the Drinfeld realization of $U_{q}(\widehat{sl_{2}})$ [6][7]. By this fact, we can understand (2.17) as a kind of quantum group structure. By this way, we can define the Ding-Iohara algebra $\mathcal{U}(q,t)$ as follows [17].

\begin{definition}[\textbf{Ding-Iohara algebra $\mathcal{U}(q,t)$}]
Let $g(x)$ be the same function defined by (2.16). Let $\gamma$ be the central, invertible element and set currents 
$x^{\pm}(z):=\sum_{n\in{\mathbf{Z}}}x^{\pm}_{n}z^{-n}$, $\psi^{\pm}(z):=\sum_{n\in\mathbf{Z}}\psi^{\pm}_{n}z^{-n}$ satisfying the relations :
\begin{align}
&\hskip 1.5cm [\psi^{\pm}(z), \psi^{\pm}(w)]=0, \quad \psi^{+}(z)\psi^{-}(w)=\frac{g(\gamma z/w)}{g(\gamma^{-1}z/w)}\psi^{-}(w)\psi^{+}(z), \notag\\
&\psi^{\pm}(z)x^{+}(w)=g\left(\gamma^{\pm\frac{1}{2}}\frac{z}{w}\right)x^{+}(w)\psi^{\pm}(z), \quad \psi^{\pm}(z)x^{-}(w)=g\left(\gamma^{\mp\frac{1}{2}}\frac{z}{w}\right)^{-1}x^{-}(w)\psi^{\pm}(z), \notag\\
&\hskip 4cm x^{\pm}(z)x^{\pm}(w)=g\left(\frac{z}{w}\right)^{\pm 1}x^{\pm}(w)x^{\pm}(z), \notag\\
&[x^{+}(z),x^{-}(w)]=\frac{(1-q)(1-t^{-1})}{1-qt^{-1}}\bigg\{\delta\Big(\gamma\frac{w}{z}\Big)\psi^{+}\big(\gamma^{1/2}w\big)-\delta\Big(\gamma^{-1}\frac{w}{z}\Big)\psi^{-}\big(\gamma^{-1/2}w\big)\bigg\}.
\end{align}
Then we define the Ding-Iohara algebra $\mathcal{U}(q,t)$ to be the associative $\mathbf{C}$-algebra generated by $\{x^{\pm}_{n}\}_{n\in{\mathbf{Z}}},\,\{\psi^{\pm}_{n}\}_{n\in{\mathbf{Z}}}$ and $\gamma$ with the above relations.
\end{definition}

Due to the proposition 2.8, the map 
\begin{align*}
\gamma \mapsto (qt^{-1})^{-1/2}, \quad x^{+}(z) \mapsto \eta(z), \quad x^{-}(z) \mapsto \xi(z), \quad \psi^{\pm}(z) \mapsto \varphi^{\pm}(z)
\end{align*}
gives a representation of the Ding-Iohara algebra (the free field realization).

\begin{remark}
It is known that $\mathcal{U}(q,t)$ has the coproduct $\Delta : \mathcal{U}(q,t) \to \mathcal{U}(q,t)\otimes \mathcal{U}(q,t)$ defined as follows [17] :
\begin{align}
&\Delta(\gamma^{\pm 1})=\gamma^{\pm 1}\otimes\gamma^{\pm 1}, \quad \Delta(\psi^{\pm}(z))=\psi^{\pm}(\gamma^{\pm1/2}_{(2)}z)\otimes\psi^{\pm}(\gamma^{\mp1/2}_{(1)}z), \notag\\
&\Delta(x^{+}(z))=x^{+}(z)\otimes 1+\psi^{-}(\gamma^{1/2}_{(1)}z)\otimes x^{+}(\gamma_{(1)}z), \notag\\
&\Delta(x^{-}(z))=x^{-}(\gamma_{(2)}z)\otimes \psi^{+}(\gamma^{1/2}_{(2)}z)+1\otimes x^{-}(z).
\end{align}
Here we define $\gamma_{(1)}:=\gamma\otimes 1,\, \gamma_{(2)}:=1\otimes\gamma$.
\end{remark}

\section{Elliptic Ding-Iohara algebra}
In this section, we are going to show that : 1) From the elliptic kernel function we can construct elliptic currents, 2) From relations among the elliptic currents satisfy, an elliptic analog of the Ding-Iohara algebra arises. 

In the following, we use parameters $q,t,p\in{\mathbf{C}}$ which satisfy $|q|<1,\,|p|<1$.

\subsection{Kernel function introduced by Komori, Noumi and Shiraishi}
First fix a positive integer $N\in\mathbf{Z}_{>0}$. The elliptic kernel function introduced by Komori, Noumi and Shiraishi [18] is defined as 
\begin{align}
\Pi(q,t,p)(x,y):=\prod_{1\leq{i,j}\leq{N}}\frac{\Gamma_{q,p}(x_{i}y_{j})}{\Gamma_{q,p}(tx_{i}y_{j})}.
\end{align}
As $\Gamma_{q,p}(x) \xrightarrow[p \to 0]{} (x;q)_{\infty}^{-1}$, the elliptic kernel function degenerates to $\Pi(q,t)(x,y)$ in the limit $p \to 0$ :
\begin{align*}
\Pi(q,t,p)(x,y) \xrightarrow[p \to 0]{} \Pi(q,t)(x,y)=\prod_{1\leq{i,j}\leq{N}}\frac{(tx_{i}y_{j};q)_{\infty}}{(x_{i}y_{j};q)_{\infty}}.
\end{align*}

\begin{remark}
In the paper [18], it is shown that the elliptic kernel function $\Pi(q,t,p)(x,y)$ and the elliptic Macdonald operator $H_{N}(q,t,p)$ in (1.3) satisfy the following relation :
\begin{align*}
H_{N}(q,t,p)_{x}\Pi(q,t,p)(x_{1},\cdots,x_{N},y_{1},\cdots,y_{N})=H_{N}(q,t,p)_{y}\Pi(q,t,p)(x_{1},\cdots,x_{N},y_{1},\cdots,y_{N}).
\end{align*}
\end{remark}

We can check the expression of $\Gamma_{q,p}(x)$ :
\begin{align*}
\Gamma_{q,p}(x)=\exp\bigg(-\sum_{n>0}\frac{(qp)^{n}}{(1-q^{n})(1-p^{n})}\frac{x^{-n}}{n}\bigg)\exp\bigg(\sum_{n>0}\frac{1}{(1-q^{n})(1-p^{n})}\frac{x^{n}}{n}\bigg).
\end{align*}
Then we can rewrite $\Pi(q,t,p)(x,y)$ by using power sums as
\begin{align}
&\Pi(q,t,p)(x,y)=\exp\bigg(\sum_{n>0}\frac{(1-t^{n})(qt^{-1}p)^{n}}{(1-q^{n})(1-p^{n})}\frac{p_{n}(\overline{x})p_{n}(\overline{y})}{n}\bigg) \notag\\
&\hskip 5cm \exp\bigg(\sum_{n>0}\frac{1-t^{n}}{(1-q^{n})(1-p^{n})}\frac{p_{n}(x)p_{n}(y)}{n}\bigg).
\end{align}
Here $p_{n}(\overline{x}):=\sum_{i=1}^{N}x_{i}^{-n} \, (n\in{\mathbf{Z}_{>0}})$ denotes the negative power sum, and for a partition $\lambda$, set $p_{\lambda}(\overline{x}):=p_{\lambda_{1}}(\overline{x})\cdots p_{\lambda_{\ell(\lambda)}}(\overline{x})$. We also define 
\begin{align}
z_{\lambda}(q,t,p):=z_{\lambda}\prod_{i=1}^{\ell(\lambda)}(1-p^{\lambda_{i}})\frac{1-q^{\lambda_{i}}}{1-t^{\lambda_{i}}}, \quad \overline{z}_{\lambda}(q,t,p):=z_{\lambda}\prod_{i=1}^{\ell(\lambda)}\frac{1-p^{\lambda_{i}}}{(qt^{-1}p)^{\lambda_{i}}}\frac{1-q^{\lambda_{i}}}{1-t^{\lambda_{i}}}.
\end{align}
Then we can expand $\Pi(q,t,p)(x,y)$ as the following form :
\begin{align}
\Pi(q,t,p)(x,y)=\sum_{\lambda\in{\mathcal{P}}}\frac{1}{\overline{z}_{\lambda}(q,t,p)}p_{\lambda}(\overline{x})p_{\lambda}(\overline{y})\sum_{\mu\in{\mathcal{P}}}\frac{1}{z_{\mu}(q,t,p)}p_{\mu}(x)p_{\mu}(y).
\end{align}

\subsection{Operator $\phi(p;z)$ and elliptic currents $\eta(p;z),\,\xi(p;z)$ and $\varphi^{\pm}(p;z)$}
Here in this subsection we are going to define the elliptic currents and study their properties. Keeping the expression of $\Pi(q,t,p)(x,y)$ as (3.4) in mind, we introduce  an algebra $\mathcal{B}_{a,b}$ of bosons generated by $\{a_{n}\}_{n\in{\mathbf{Z}\setminus\{0\}}},\, \{b_{n}\}_{n\in{\mathbf{Z}\setminus\{0\}}}$ and the following relations :
\begin{align}
&[a_{m},a_{n}]=m(1-p^{|m|})\frac{1-q^{|m|}}{1-t^{|m|}}\delta_{m+n,0}, \quad [b_{m},b_{n}]=m\frac{1-p^{|m|}}{(qt^{-1}p)^{|m|}}\frac{1-q^{|m|}}{1-t^{|m|}}\delta_{m+n,0}, \notag\\
&[a_{m},b_{n}]=0.
\end{align}
As in the trigonometric case, let $|0 \rangle$ be the vacuum vector which satisfies the conditions $a_{n}|0 \rangle=b_{n}|0 \rangle=0 \, (n>0)$ and set the boson Fock space $\mathcal{F}$ as a left $\mathcal{B}_{a,b}$ module :
\begin{align}
\mathcal{F}:=\text{span}\{a_{-\lambda}b_{-\mu}|0 \rangle : \lambda, \mu\in{\mathcal{P}}\}.
\end{align}
The dual vacuum vector $\langle 0|$ is defined by the conditions $\langle 0|a_{n}=\langle 0|b_{n}=0 \, (n<0)$ and we set the dual boson Fock space $\mathcal{F}^{\ast}$ as a right $\mathcal{B}_{a,b}$ module :
\begin{align}
\mathcal{F}^{\ast}:=\text{span}\{\langle 0|a_{\lambda}b_{\mu} : \lambda, \mu\in{\mathcal{P}}\}.
\end{align}
We define a bilinear form $\langle \bullet|\bullet \rangle : \mathcal{F}^{\ast} \times \mathcal{F} \to \mathbf{C}$ by the following conditions.
\begin{align*}
(1) \, \langle 0|0 \rangle=1, \quad (2) \, \langle 0|a_{\lambda_{1}}b_{\lambda_{2}}a_{-\mu_{1}}b_{-\mu_{2}}|0 \rangle=\delta_{\lambda_{1}\mu_{1}}\delta_{\lambda_{2}\mu_{2}}z_{\lambda_{1}}(q,t,p)\overline{z}_{\lambda_{2}}(q,t,p).
\end{align*}
We also define the normal ordering $\bm{:} \bullet \bm{:}$ as usual :
\begin{align*}
\bm{:}a_{m}a_{n}\bm{:}=
\begin{cases}
a_{m}a_{n} \quad (m<n), \\
a_{n}a_{m} \quad (m\geq{n}),
\end{cases}
\bm{:}b_{m}b_{n}\bm{:}=
\begin{cases}
b_{m}b_{n} \quad (m<n), \\
b_{n}b_{m} \quad (m\geq{n}).
\end{cases}
\end{align*}

\begin{remark}
The above defined algebra of bosons leads to consider the space of symmetric functions $\Lambda_{N}(q,t,p):=\mathbf{C}[[x_{i},x_{i}^{-1} : 1\leq{i}\leq{N}]]^{\mathfrak{S}_{N}}$.
But it is not clear whether an elliptic analog of the Macdonald symmetric functions live in $\Lambda_{N}(q,t,p)$.
\end{remark}

Define operators $\phi(p;z) : \mathcal{F} \to \mathcal{F}\otimes \mathbf{C}[[z,z^{-1}]]$ and $\phi^{\ast}(p;z) : \mathcal{F}^{\ast} \to \mathcal{F}^{\ast}\otimes\mathbf{C}[[z,z^{-1}]]$ as follows :
\begin{align}
&\phi(p;z):=\exp\bigg(\sum_{n>0}\frac{(1-t^{n})(qt^{-1}p)^{n}}{(1-q^{n})(1-p^{n})}b_{-n}\frac{z^{-n}}{n}\bigg)\exp\bigg(\sum_{n>0}\frac{1-t^{n}}{(1-q^{n})(1-p^{n})}a_{-n}\frac{z^{n}}{n}\bigg), \\
&\phi^{\ast}(p;z):=\exp\bigg(\sum_{n>0}\frac{(1-t^{n})(qt^{-1}p)^{n}}{(1-q^{n})(1-p^{n})}b_{n}\frac{z^{-n}}{n}\bigg)\exp\bigg(\sum_{n>0}\frac{1-t^{n}}{(1-q^{n})(1-p^{n})}a_{n}\frac{z^{n}}{n}\bigg).
\end{align}
Set $\phi_{N}(p;x):=\prod_{j=1}^{N}\phi(p;x_{j})$, $\phi^{\ast}_{N}(p;x):=\prod_{j=1}^{N}\phi^{\ast}(p;x_{j})$ $(N\in\mathbf{Z}_{>0})$, then we have 
\begin{align}
\langle 0|\phi^{\ast}_{N}(p;x)\phi_{N}(p;y)|0 \rangle=\Pi(q,t,p)(x_{1},\cdots,x_{N},y_{1},\cdots,y_{N}).
\end{align}
We can check this. First we have
\begin{align*}
&\quad \phi^{\ast}(p;z)\phi(p;w) \\
&=\exp\bigg(\sum_{m>0}\frac{(1-t^{m})(qt^{-1}p)^{m}}{(1-q^{m})(1-p^{m})}\frac{(1-t^{m})(qt^{-1}p)^{m}}{(1-q^{m})(1-p^{m})}\cdot m\frac{1-p^{m}}{(qt^{-1}p)^{m}}\frac{1-q^{m}}{1-t^{m}}\frac{(zw)^{-m}}{m}\bigg) \\
&\times \exp\bigg(\sum_{m>0}\frac{(1-t^{m})}{(1-q^{m})(1-p^{m})}\frac{(1-t^{m})}{(1-q^{m})(1-p^{m})}\cdot m(1-p^{m})\frac{1-q^{m}}{1-t^{m}}\frac{(zw)^{m}}{m}\bigg) \\
&\hskip 11cm \times \bm{:}\phi^{\ast}(p;z)\phi(p;w)\bm{:} \\
&=\exp\bigg(\sum_{m>0}\frac{(1-t^{m})(qt^{-1}p)^{m}}{(1-q^{m})(1-p^{m})}\frac{(zw)^{-m}}{m}\bigg)\exp\bigg(\sum_{m>0}\frac{(1-t^{m})}{(1-q^{m})(1-p^{m})}\frac{(zw)^{m}}{m}\bigg) \\
&\hskip 11cm \times \bm{:}\phi^{\ast}(p;z)\phi(p;w)\bm{:}.
\end{align*}
By this equation and the expression of the kernel function (3.2), we have (3.10). 

Next let us construct an operator $\eta(p;z) : \mathcal{F} \to \mathcal{F}\otimes \mathbf{C}[[z,z^{-1}]]$ which satisfies the conditions :
\begin{align}
1) \, \bm{:}\eta(p;tz)\phi(p;z)\bm{:}|0 \rangle=\phi(p;qz)|0 \rangle, \quad 2) \, \langle 0|\bm{:}\phi^{\ast}(p;z)\eta(p;z^{-1})\bm{:}=\langle 0|\phi^{\ast}(p;qz).
\end{align}
These conditions are satisfied by the following operator, which we would like to call the elliptic current.

\begin{prop}[\textbf{Elliptic current $\eta(p;z)$}] 
Let $\eta(p;z) : \mathcal{F} \to \mathcal{F}\otimes \mathbf{C}[[z,z^{-1}]]$ be an operator defined as follows :
\begin{align}
\eta(p;z):=\bm{:}\exp\bigg(-\sum_{n\neq{0}}\frac{1-t^{-n}}{1-p^{|n|}}p^{|n|}b_{n}\frac{z^{n}}{n}\bigg)\exp\bigg(-\sum_{n\neq{0}}\frac{1-t^{n}}{1-p^{|n|}}a_{n}\frac{z^{-n}}{n}\bigg)\bm{:}.
\end{align}
Then we have
\begin{align}
&\text{i) \, \textit{$\eta(p;z)$ satisfies the conditions 1) and 2) in (3.11).}} \notag\\
&\text{ii)} \quad \eta(p;z)\eta(p;w)=\frac{\Theta_{p}(w/z)\Theta_{p}(qt^{-1}w/z)}{\Theta_{p}(qw/z)\Theta_{p}(t^{-1}w/z)}\bm{:}\eta(p;z)\eta(p;w)\bm{:} \quad 
\left(\begin{cases}
\smallskip |p|<|qw/z|<1, \\ 
\smallskip |p|<|t^{-1}w/z|<1
\end{cases}\right).
\end{align}
\end{prop}

\begin{proof}[\textit{Proof}]
 i) We show 1) in (3.11) is satisfied. Then we have
\begin{align*}
\bm{:}\eta(p;tz)\phi(p;z)\bm{:}|0 \rangle=(\eta(p;tz))_{-}\phi(p;z)| 0\rangle.
\end{align*}
Hence what we have to show is $(\eta(p;tz))_{-}\phi(p;z)=\phi(p;qz)$. The proof of the relation is straightforward. Since the operator $(\eta(p;z))_{-}$ takes the form 
\begin{align*}
(\eta(p;z))_{-}=\exp\bigg(\sum_{n>0}\frac{1-t^{n}}{1-p^{n}}p^{n}b_{-n}\frac{z^{-n}}{n}\bigg)\exp\bigg(\sum_{n>0}\frac{1-t^{-n}}{1-p^{n}}a_{-n}\frac{z^{n}}{n}\bigg),
\end{align*}
we have
\begin{align*}
&\quad (\eta(p;tz))_{-}\phi(p;z) \\
&=\exp\bigg(\sum_{n>0}\frac{(qt^{-1}p)^{n}(1-t^{n})}{(1-q^{n})(1-p^{n})}\{t^{-n}(1-q^{n})(qt^{-1})^{-n}+1\}b_{-n}\frac{z^{-n}}{n}\bigg) \\
&\quad \times\exp\bigg(\sum_{n>0}\frac{1-t^{n}}{(1-q^{n})(1-p^{n})}\{-t^{-n}(1-q^{n})t^{n}+1\}a_{-n}\frac{z^{n}}{n}\bigg) \\
&=\exp\bigg(\sum_{n>0}\frac{(qt^{-1}p)^{n}(1-t^{n})}{(1-q^{n})(1-p^{n})}q^{-n}b_{-n}\frac{z^{-n}}{n}\bigg)\exp\bigg(\sum_{n>0}\frac{1-t^{n}}{(1-q^{n})(1-p^{n})}q^{n}a_{-n}\frac{z^{n}}{n}\bigg) \\
&=\phi(p;qz). 
\end{align*}
Next we show $\eta(p;z)$ satisfies 2) in (3.11). Due to the relation
\begin{align*}
\langle 0|\bm{:}\phi^{\ast}(p;z)\eta(p;z^{-1})\bm{:}=\langle 0|\phi^{\ast}(p;z)(\eta(p;z^{-1}))_{+},
\end{align*}
what we have to show is $\phi^{\ast}(p;z)(\eta(p;z^{-1}))_{+}=\phi^{\ast}(p;qz)$. $(\eta(p;z^{-1}))_{+}$ takes the form as
\begin{align*}
(\eta(p;z^{-1}))_{+}=\exp\bigg(-\sum_{n>0}\frac{1-t^{-n}}{1-p^{n}}p^{n}b_{n}\frac{z^{-n}}{n}\bigg)\exp\bigg(-\sum_{n>0}\frac{1-t^{n}}{1-p^{n}}a_{n}\frac{z^{n}}{n}\bigg),
\end{align*}
hence we have
\begin{align*}
\phi^{\ast}(p;z)(\eta(p;z^{-1}))_{+}
&=\exp\bigg(\sum_{n>0}\frac{(1-t^{n})(qt^{-1}p)^{n}}{(1-q^{n})(1-p^{n})}\{1+q^{-n}(1-q^{n})\}b_{n}\frac{z^{-n}}{n}\bigg) \\
&\quad \times \exp\bigg(\sum_{n>0}\frac{1-t^{n}}{(1-q^{n})(1-p^{n})}\{1-(1-q^{n})\}a_{n}\frac{z^{n}}{n}\bigg) \\
&=\phi^{\ast}(p;qz).
\end{align*}

ii) By Wick's theorem, we have the following :
\begin{align*}
&\quad \eta(p;z)\eta(p;w)\\
&=\exp\bigg(\sum_{m>0}\frac{1-t^{-m}}{1-p^{m}}p^{m}\frac{1-t^{m}}{1-p^{m}}p^{m}\cdot m\frac{1-p^{m}}{(qt^{-1}p)^{m}}\frac{1-q^{m}}{1-t^{m}}\frac{(z/w)^{m}}{m(-m)}\bigg)\\
&\quad \times\exp\bigg(\sum_{m>0}\frac{1-t^{m}}{1-p^{m}}\frac{1-t^{-m}}{1-p^{m}}\cdot m(1-p^{m})\frac{1-q^{m}}{1-t^{m}}\frac{(w/z)^{m}}{m(-m)}\bigg)\bm{:}\eta(p;z)\eta(p;w)\bm{:}\\
&=\exp\bigg(-\sum_{m>0}\frac{(1-q^{m})(1-t^{-m})(qt^{-1})^{-m}}{1-p^{m}}p^{m}\frac{(z/w)^{m}}{m}\bigg)\\
&\quad \times\exp\bigg(-\sum_{m>0}\frac{(1-q^{m})(1-t^{-m})}{1-p^{m}}\frac{(w/z)^{m}}{m}\bigg)\bm{:}\eta(p;z)\eta(p;w)\bm{:}\\
&=\frac{(q^{-1}tpz/w;p)_{\infty}(pz/w;p)_{\infty}}{(tpz/w;p)_{\infty}(q^{-1}pz/w;p)_{\infty}}\frac{(w/z;p)_{\infty}(qt^{-1}w/z;p)_{\infty}}{(qw/z;p)_{\infty}(t^{-1}w/z;p)_{\infty}}\bm{:}\eta(p;z)\eta(p;w)\bm{:}\\
&=\frac{\Theta_{p}(w/z)\Theta_{p}(qt^{-1}w/z)}{\Theta_{p}(qw/z)\Theta_{p}(t^{-1}w/z)}\bm{:}\eta(p;z)\eta(p;w)\bm{:}. \quad \Box
\end{align*}
\end{proof}

Since the relation (3.13) is an elliptic analog of the trigonometric case (2.10), we can understand $\eta(p;z)$ is an elliptic analog of $\eta(z)$. In the similar way, we can define an operator $\xi(p;z) : \mathcal{F} \to \mathcal{F}\otimes \mathbf{C}[[z,z^{-1}]]$ which is an elliptic analog of $\xi(z)$.

\begin{prop}[\textbf{Elliptic current $\xi(p;z)$}]
Let $\xi(p;z) : \mathcal{F} \to \mathcal{F}\otimes \mathbf{C}[[z,z^{-1}]]$ be an operator defined as follows :
\begin{align}
\xi(p;z)
:=\bm{:}\exp\bigg(\sum_{n\neq{0}}\frac{1-t^{-n}}{1-p^{|n|}}\gamma^{-|n|}p^{|n|}b_{n}\frac{z^{n}}{n}\bigg)\exp\bigg(\sum_{n\neq{0}}\frac{1-t^{n}}{1-p^{|n|}}\gamma^{|n|}a_{n}\frac{z^{-n}}{n}\bigg)\bm{:}.
\end{align}
Then $\xi(p;z)$ satisfies
\begin{align}
&\text{i)} \quad (\xi(p;\gamma z))_{-}\phi(p;z)=\phi(p;q^{-1}z), \quad \phi^{\ast}(p;z)(\xi(p;t\gamma^{-1}z^{-1}))_{+}=\phi^{\ast}(p;q^{-1}z), \\
&\text{ii)} \quad \xi(p;z)\xi(p;w)=\frac{\Theta_{p}(w/z)\Theta_{p}(q^{-1}tw/z)}{\Theta_{p}(q^{-1}w/z)\Theta_{p}(tw/z)}\bm{:}\xi(p;z)\xi(p;w)\bm{:} \quad
\left(\begin{cases}
\smallskip |p|<|q^{-1}w/z|<1, \\ 
\smallskip |p|<|tw/z|<1
\end{cases}\right).
\end{align}
\end{prop}

\begin{proof}[\textit{Proof}]
Since the proof of (3.15) is quite similar to the proof of i) in the proposition 3.3, we omit it. (3.16) is shown as follows :
\begin{align*}
&\quad \xi(p;z)\xi(p;w) \\
&=\exp\bigg(\sum_{m>0}\frac{1-t^{-m}}{1-p^{m}}\gamma^{-m}p^{m}\frac{1-t^{m}}{1-p^{m}}\gamma^{-m}p^{m}\cdot m\frac{1-p^{m}}{(qt^{-1}p)^{m}}\frac{1-q^{m}}{1-t^{m}}\frac{(z/w)^{m}}{m(-m)}\bigg)\\
&\quad\times\exp\bigg(\sum_{m>0}\frac{1-t^{m}}{1-p^{m}}\gamma^{m}\frac{1-t^{-m}}{1-p^{m}}\gamma^{m}\cdot m(1-p^{m})\frac{1-q^{m}}{1-t^{m}}\frac{(w/z)^{m}}{m(-m)}\bigg)\bm{:}\xi(p;z)\xi(p;w)\bm{:}\\
&=\exp\bigg(-\sum_{m>0}\frac{(1-q^{m})(1-t^{-m})}{1-p^{m}}p^{m}\frac{(z/w)^{m}}{m}\bigg)\\
&\quad \times\exp\bigg(-\sum_{m>0}\frac{(1-q^{m})(1-t^{-m})(qt^{-1})^{-m}}{1-p^{m}}\frac{(w/z)^{m}}{m}\bigg)\bm{:}\xi(p;z)\xi(p;w)\bm{:}\\
&=\frac{\Theta_{p}(w/z)\Theta_{p}(q^{-1}tw/z)}{\Theta_{p}(q^{-1}w/z)\Theta_{p}(tw/z)}\bm{:}\xi(p;z)\xi(p;w)\bm{:}. \quad \Box
\end{align*}
\end{proof}

As in the trigonometric case, it is natural to calculate a commutation relation between $\eta(p;z)$ and $\xi(p;z)$. For the calculation of $[\eta(p;z),\xi(p;w)]$, we need a lemma which gives a relation between the theta function $\Theta_{p}(x)$ and the delta function $\delta(x)$.

\begin{lemma}
For the theta function $\Theta_{p}(x)$ and the delta function $\delta(x)$, the following relations are satisfied :
\begin{align}
&\frac{1}{\Theta_{p}(x)}+\frac{x^{-1}}{\Theta_{p}(x^{-1})}=\frac{1}{(p;p)_{\infty}^{3}}\delta(x), \\
&\frac{1}{\Theta_{p}(x)}+\frac{x^{-1}}{\Theta_{p}(px)}=\frac{1}{(p;p)_{\infty}^{3}}\delta(x).
\end{align}
This leads that
\begin{align}
\frac{1}{\Theta_{p}(px)}=\frac{1}{\Theta_{p}(x^{-1})}.
\end{align}
\end{lemma}

\begin{proof}[\textit{Proof}]
To prove the relation (3.17), let us recall the formal expression of the delta function as
\begin{align*}
\delta(x)=\sum_{n\in{\mathbf{Z}}}x^{n}=\frac{1}{1-x}+\frac{x^{-1}}{1-x^{-1}}.
\end{align*}
By this expression, we have the following :
\begin{align*}
\frac{1}{\Theta_{p}(x)}&=\frac{1}{(p;p)_{\infty}(x;p)_{\infty}(px^{-1};p)_{\infty}}\\
&=\frac{1}{(p;p)_{\infty}}\frac{1}{(1-x)(px;p)_{\infty}(px^{-1};p)_{\infty}}\\
&=\frac{1}{(p;p)_{\infty}}\Big(\delta(x)-\frac{x^{-1}}{1-x^{-1}}\Big)\frac{1}{(px;p)_{\infty}(px^{-1};p)_{\infty}}\\
&=\frac{1}{(p;p)_{\infty}^{3}}\delta(x)-\frac{x^{-1}}{\Theta_{p}(x^{-1})}. 
\end{align*}
The relation (3.18) is shown in the similar way :
\begin{align*}
\frac{1}{\Theta_{p}(px)}&=\frac{1}{(p;p)_{\infty}(px;p)_{\infty}(x^{-1};p)_{\infty}} \\
&=\frac{1}{(p;p)_{\infty}(px;p)_{\infty}}\frac{1}{(1-x^{-1})(px^{-1};p)_{\infty}} \\
&=\frac{1}{(p;p)_{\infty}(px;p)_{\infty}}\Big(\delta(x)-\frac{x}{1-x}\Big)\frac{1}{(px^{-1};p)_{\infty}} \\
&=\frac{1}{(p;p)_{\infty}^{3}}\delta(x)-\frac{x}{\Theta_{p}(x)}.
\end{align*}
By the subtraction $(3.17)-(3.18)$, we have $1/\Theta_{p}(px)=1/\Theta_{p}(x^{-1})$. \quad $\Box$ 
\end{proof}

\begin{remark}
The relation (3.17), (3.18) should be also recognized in the context of the Sato hyperfunction [5].
\end{remark}

From this lemma, we can calculate $[\eta(p;z),\xi(p;w)]$ as follows.

\begin{prop}[\textbf{Commutator $[\eta(p;z),\xi(p;w)]$}] 
Let $\varphi^{\pm}(p;z) : \mathcal{F} \to \mathcal{F}\otimes \mathbf{C}[[z,z^{-1}]]$ be operators defined as
\begin{align}
\varphi^{+}(p;z):=\bm{:}\eta(p;\gamma^{1/2}z)\xi(p;\gamma^{-1/2}z)\bm{:}, \quad \varphi^{-}(p;z):=\bm{:}\eta(p;\gamma^{-1/2}z)\xi(p;\gamma^{1/2}z)\bm{:}.
\end{align}
Then the relation holds :
\begin{align}
[\eta(p;z),\xi(p;w)]=\frac{\Theta_{p}(q)\Theta_{p}(t^{-1})}{(p;p)_{\infty}^{3}\Theta_{p}(qt^{-1})}\bigg\{\delta\Big(\gamma\frac{w}{z}\Big)\varphi^{+}(p;\gamma^{1/2}w)-\delta\Big(\gamma^{-1}\frac{w}{z}\Big)\varphi^{-}(p;\gamma^{-1/2}w)\bigg\}.
\end{align}
\end{prop}

\begin{proof}[\textit{Proof}]
By Wick's theorem, we have the following :
\begin{align*}
&\quad \eta(p;z)\xi(p;w) \\
&=\exp\bigg(-\sum_{m>0}\frac{1-t^{-m}}{1-p^{m}}p^{m}\frac{1-t^{m}}{1-p^{m}}\gamma^{-m}p^{m}\cdot m\frac{1-p^{m}}{(qt^{-1}p)^{m}}\frac{1-q^{m}}{1-t^{m}}\frac{(z/w)^{m}}{m(-m)}\bigg) \\
&\quad\times\exp\bigg(-\sum_{m>0}\frac{1-t^{m}}{1-p^{m}}\frac{1-t^{-m}}{1-p^{m}}\gamma^{m}\cdot m(1-p^{m})\frac{1-q^{m}}{1-t^{m}}\frac{(w/z)^{m}}{m(-m)}\bigg)\bm{:}\eta(p;z)\xi(p;w)\bm{:}\\
&=\exp\bigg(\sum_{m>0}\frac{(1-q^{m})(1-t^{-m})}{1-p^{m}}\gamma^{m}p^{m}\frac{(z/w)^{m}}{m}\bigg)\\
&\quad \times\exp\bigg(\sum_{m>0}\frac{(1-q^{m})(1-t^{-m})}{1-p^{m}}\gamma^{m}\frac{(w/z)^{m}}{m}\bigg)\bm{:}\eta(p;z)\xi(p;w)\bm{:}\\
&=\frac{\Theta_{p}(q\gamma w/z)\Theta_{p}(q^{-1}\gamma^{-1}w/z)}{\Theta_{p}(\gamma w/z)\Theta_{p}(\gamma^{-1}w/z)}\bm{:}\eta(p;z)\xi(p;w)\bm{:}, \\
\end{align*}
\begin{align*}
&\quad \xi(p;w)\eta(p;z) \\
&=\exp\bigg(-\sum_{m>0}\frac{1-t^{-m}}{1-p^{m}}\gamma^{-m}p^{m}\frac{1-t^{m}}{1-p^{m}}p^{m}\cdot m\frac{1-p^{m}}{(qt^{-1}p)^{m}}\frac{1-q^{m}}{1-t^{m}}\frac{(w/z)^{m}}{m(-m)}\bigg) \\
&\quad\times\exp\bigg(-\sum_{m>0}\frac{1-t^{m}}{1-p^{m}}\gamma^{m}\frac{1-t^{-m}}{1-p^{m}}\cdot m(1-p^{m})\frac{1-q^{m}}{1-t^{m}}\frac{(z/w)^{m}}{m(-m)}\bigg)\bm{:}\eta(p;z)\xi(p;w)\bm{:}\\
&=\exp\bigg(\sum_{m>0}\frac{(1-q^{m})(1-t^{-m})}{1-p^{m}}\gamma^{m}p^{m}\frac{(w/z)^{m}}{m}\bigg)\\
&\quad \times\exp\bigg(\sum_{m>0}\frac{(1-q^{m})(1-t^{-m})}{1-p^{m}}\gamma^{m}\frac{(z/w)^{m}}{m}\bigg)\bm{:}\eta(p;z)\xi(p;w)\bm{:}\\
&=\frac{\Theta_{p}(q\gamma z/w)\Theta_{p}(q^{-1}\gamma^{-1}z/w)}{\Theta_{p}(\gamma z/w)\Theta_{p}(\gamma^{-1}z/w)}\bm{:}\eta(p;z)\xi(p;w)\bm{:}.
\end{align*}
Then we have
\begin{align*}
&\quad [\eta(p;z),\xi(p;w)] \\
&=\bigg\{\frac{\Theta_{p}(q\gamma w/z)\Theta_{p}(q^{-1}\gamma^{-1}w/z)}{\Theta_{p}(\gamma w/z)\Theta_{p}(\gamma^{-1}w/z)}-\frac{\Theta_{p}(q\gamma z/w)\Theta_{p}(q^{-1}\gamma^{-1}z/w)}{\Theta_{p}(\gamma z/w)\Theta_{p}(\gamma^{-1}z/w)}\bigg\}\bm{:}\eta(p;z)\xi(p;w)\bm{:} \\
&=\Theta_{p}(q\gamma w/z)\Theta_{p}(q^{-1}\gamma^{-1}w/z)\bigg\{\frac{1}{\Theta_{p}(\gamma w/z)\Theta_{p}(\gamma^{-1}w/z)}-\frac{(z/w)^{2}}{\Theta_{p}(\gamma z/w)\Theta_{p}(\gamma^{-1}z/w)}\bigg\}\\
&\hskip 12cm \times\bm{:}\eta(p;z)\xi(p;w)\bm{:}.
\end{align*}
From the lemma 3.5, the following holds :
\begin{align*}
&\quad \frac{1}{\Theta_{p}(\gamma x)\Theta_{p}(\gamma^{-1}x)}-\frac{x^{-2}}{\Theta_{p}(\gamma x^{-1})\Theta_{p}(\gamma^{-1}x^{-1})} \\
&=\bigg\{\frac{1}{\Theta_{p}(\gamma x)}+\frac{\gamma^{-1}x^{-1}}{\Theta_{p}(\gamma^{-1}x^{-1})}\bigg\}\frac{1}{\Theta_{p}(\gamma^{-1}x)}
-\frac{\gamma^{-1}x^{-1}}{\Theta_{p}(\gamma^{-1}x^{-1})}\bigg\{\frac{1}{\Theta_{p}(\gamma^{-1}x)}+\frac{\gamma x^{-1}}{\Theta_{p}(\gamma x^{-1})}\bigg\} \\
&=\frac{1}{(p;p)_{\infty}^{3}}\delta(\gamma x)\frac{1}{\Theta_{p}(\gamma^{-1}x)}-\frac{\gamma^{-1}x^{-1}}{\Theta_{p}(\gamma^{-1}x^{-1})}\frac{1}{(p;p)_{\infty}^{3}}\delta(\gamma^{-1}x) \\
&=\frac{1}{(p;p)_{\infty}^{3}\Theta_{p}(qt^{-1})}\{\delta(\gamma x)-\gamma^{-2}\delta(\gamma^{-1}x)\}.
\end{align*}
This leads that
\begin{align*}
\frac{\Theta_{p}(q\gamma x)\Theta_{p}(q^{-1}\gamma^{-1}x)}{\Theta_{p}(\gamma x)\Theta_{p}(\gamma^{-1}x)}-\frac{\Theta_{p}(q\gamma x^{-1})\Theta_{p}(q^{-1}\gamma^{-1}x^{-1})}{\Theta_{p}(\gamma x^{-1})\Theta_{p}(\gamma^{-1}x^{-1})}
=\frac{\Theta_{p}(q)\Theta_{p}(t^{-1})}{(p;p)_{\infty}^{3}\Theta_{p}(qt^{-1})}\{\delta(\gamma x)-\delta(\gamma^{-1}x)\}.
\end{align*}
By this relation and the definition of $\varphi^{\pm}(p;z)$, we have (3.21). \quad $\Box$
\end{proof}

The commutation relation (3.21) is also an elliptic analog of the trigonometric case in (2.17). Furthermore, we can show the following theorem by Wick's theorem.

\begin{theorem}[\textbf{Relations of $\eta(p;z),\,\xi(p;z)$ and $\varphi^{\pm}(p;z)$}] 
We define the structure function $g_{p}(x)$ as
\begin{align}
g_{p}(x):=\frac{\Theta_{p}(qx)\Theta_{p}(t^{-1}x)\Theta_{p}(q^{-1}tx)}{\Theta_{p}(q^{-1}x)\Theta_{p}(tx)\Theta_{p}(qt^{-1}x)}.
\end{align}
Then $\eta(p;z),\,\xi(p;z)$ and $\varphi^{\pm}(p;z)$ satisfy the relations :
\begin{align}
&\hskip 0.5cm [\varphi^{\pm}(p;z), \varphi^{\pm}(p;w)]=0, \quad \varphi^{+}(p;z)\varphi^{-}(p;w)=\frac{g_{p}(\gamma z/w)}{g_{p}(\gamma^{-1}z/w)}\varphi^{-}(p;w)\varphi^{+}(p;z),\\
&\hskip 3cm \varphi^{\pm}(p;z)\eta(p;w)=g_{p}\Big(\gamma^{\pm\frac{1}{2}}\frac{z}{w}\Big)\eta(p;w)\varphi^{\pm}(p;z),\\
&\hskip 3cm \varphi^{\pm}(p;z)\xi(p;w)=g_{p}\Big(\gamma^{\mp\frac{1}{2}}\frac{z}{w}\Big)^{-1}\xi(p;w)\varphi^{\pm}(p;z),\\
&\hskip 3cm \eta(p;z)\eta(p;w)=g_{p}\Big(\frac{z}{w}\Big)\eta(p;w)\eta(p;z),\\
&\hskip 3cm \xi(p;z)\xi(p;w)=g_{p}\Big(\frac{z}{w}\Big)^{-1}\xi(p;w)\xi(p;z),\\
&[\eta(p;z),\xi(p;w)]
=\frac{\Theta_{p}(q)\Theta_{p}(t^{-1})}{(p;p)_{\infty}^{3}\Theta_{p}(qt^{-1})}\bigg\{\delta\Big(\gamma\frac{w}{z}\Big)\varphi^{+}(p;\gamma^{1/2}w)-\delta\Big(\gamma^{-1}\frac{w}{z}\Big)\varphi^{-}(p;\gamma^{-1/2}w)\bigg\}.
\end{align}
\end{theorem}

\begin{proof}[\textit{Proof}]
Relations (3.26) and (3.27) follow from (3.13), (3.16). By the definition of $\varphi^{\pm}(p;z)$, they take forms as follows :
\begin{align}
&\quad \varphi^{+}(p;z)\notag\\
&=\exp\bigg(-\sum_{n>0}\frac{1-t^{-n}}{1-p^{n}}p^{n}(1-\gamma^{-2n})\gamma^{n/2}b_{n}\frac{z^{n}}{n}\bigg)\exp\bigg(-\sum_{n>0}\frac{1-t^{n}}{1-p^{n}}(1-\gamma^{2n})\gamma^{-n/2}a_{n}\frac{z^{-n}}{n}\bigg),\\
&\quad \varphi^{-}(p;z)\notag\\
&=\exp\bigg(-\sum_{n<0}\frac{1-t^{-n}}{1-p^{-n}}p^{-n}(1-\gamma^{2n})\gamma^{-n/2}b_{n}\frac{z^{n}}{n}\bigg)\exp\bigg(-\sum_{n<0}\frac{1-t^{n}}{1-p^{-n}}(1-\gamma^{-2n})\gamma^{n/2}a_{n}\frac{z^{-n}}{n}\bigg).
\end{align}
By these expressions, the relation $[\varphi^{\pm}(p;z),\varphi^{\pm}(p;w)]=0$ is trivial. Next we show the relation (3.23). Here we can check that $g_{p}(x)$ takes the form as follows :
\begin{align}
&\quad g_{p}(x) \notag\\
&=\exp\bigg(-\sum_{n>0}\frac{(1-q^{n})(1-t^{-n})(1-\gamma^{2n})}{1-p^{n}}p^{n}\frac{x^{-n}}{n}\bigg)\exp\bigg(\sum_{n>0}\frac{(1-q^{n})(1-t^{-n})(1-\gamma^{2n})}{1-p^{n}}\frac{x^{n}}{n}\bigg).
\end{align}
We can also check $g_{p}(x^{-1})=g_{p}(x)^{-1}$. From these facts, we have the following :
\begin{align*}
&\quad \varphi^{+}(p;z)\varphi^{-}(p;w) \\
&=\exp\bigg(\sum_{m>0}\frac{1-t^{-m}}{1-p^{m}}p^{m}(1-\gamma^{-2m})\gamma^{m/2}\frac{1-t^{m}}{1-p^{m}}p^{m}(1-\gamma^{-2m})\gamma^{m/2}m\frac{1-p^{m}}{(qt^{-1}p)^{m}}\frac{1-q^{m}}{1-t^{m}}\frac{(z/w)^{m}}{m(-m)}\bigg) \\
&\quad\times\exp\bigg(\sum_{m>0}\frac{1-t^{m}}{1-p^{m}}(1-\gamma^{2m})\gamma^{-m/2}\frac{1-t^{-m}}{1-p^{m}}(1-\gamma^{2m})\gamma^{-m/2}m(1-p^{m})\frac{1-q^{m}}{1-t^{m}}\frac{(w/z)^{m}}{m(-m)}\bigg)\\
&\hskip 13cm \times\varphi^{-}(p;w)\varphi^{+}(p;z) \\
&=\exp\bigg(-\sum_{m>0}\frac{(1-q^{m})(1-t^{-m})(1-\gamma^{2m})}{1-p^{m}}p^{m}(\gamma^{-m}-\gamma^{m})\frac{(z/w)^{m}}{m}\bigg)\\
&\quad \times\exp\bigg(\sum_{m>0}\frac{(1-q^{m})(1-t^{-m})(1-\gamma^{2m})}{1-p^{m}}(\gamma^{m}-\gamma^{-m})\frac{(w/z)^{m}}{m}\bigg)\varphi^{-}(p;w)\varphi^{+}(p;z) \\
&=\frac{g_{p}(\gamma w/z)}{g_{p}(\gamma^{-1}w/z)}\varphi^{-}(p;w)\varphi^{+}(p;z)=\frac{g_{p}(\gamma z/w)}{g_{p}(\gamma^{-1}z/w)}\varphi^{-}(p;w)\varphi^{+}(p;z) \quad( \because \, g_{p}(x^{-1})=g_{p}(x)^{-1}).
\end{align*}

Next we show the relations (3.24). By Wick's theorem, we have
\begin{align*}
&\quad \varphi^{+}(p;z)\eta(p;z) \\
&=\exp\bigg(\sum_{m>0}\frac{1-t^{-m}}{1-p^{m}}p^{m}(1-\gamma^{-2m})\gamma^{m/2}\frac{1-t^{m}}{1-p^{m}}p^{m}\cdot m\frac{1-p^{m}}{(qt^{-1}p)^{m}}\frac{1-q^{m}}{1-t^{m}}\frac{(z/w)^{m}}{m(-m)}\bigg) \\
&\quad\times\exp\bigg(\sum_{m>0}\frac{1-t^{m}}{1-p^{m}}(1-\gamma^{2m})\gamma^{-m/2}\frac{1-t^{-m}}{1-p^{m}}\cdot m(1-p^{m})\frac{1-q^{m}}{1-t^{m}}\frac{(w/z)^{m}}{m(-m)}\bigg)\eta(p;w)\varphi^{+}(p;z) \\
&=\exp\bigg(\sum_{m>0}\frac{(1-q^{m})(1-t^{-m})(1-\gamma^{2m})}{1-p^{m}}p^{m}\gamma^{m/2}\frac{(z/w)^{m}}{m}\bigg) \\
&\quad \times\exp\bigg(-\sum_{m>0}\frac{(1-q^{m})(1-t^{-m})(1-\gamma^{2m})}{1-p^{m}}\gamma^{-m/2}\frac{(w/z)^{m}}{m}\bigg)\eta(p;w)\varphi^{+}(p;z) \\
&=g_{p}\Big(\gamma^{-1/2}\frac{w}{z}\Big)^{-1}\eta(p;w)\varphi^{+}(p;z)=g_{p}\Big(\gamma^{1/2}\frac{z}{w}\Big)\eta(p;w)\varphi^{+}(p;z).
\end{align*}
Similarly, we have
\begin{align*}
&\quad \eta(p;w)\varphi^{-}(p;z) \\
&=\exp\bigg(\sum_{m>0}\frac{1-t^{-m}}{1-p^{m}}p^{m}\frac{1-t^{m}}{1-p^{m}}p^{m}(1-\gamma^{-2m})\gamma^{m/2}\cdot m\frac{1-p^{m}}{(qt^{-1}p)^{m}}\frac{1-q^{m}}{1-t^{m}}\frac{(w/z)^{m}}{m(-m)}\bigg) \\
&\quad\times\exp\bigg(\sum_{m>0}\frac{1-t^{m}}{1-p^{m}}\frac{1-t^{-m}}{1-p^{m}}(1-\gamma^{2m})\gamma^{-m/2}\cdot m(1-p^{m})\frac{1-q^{m}}{1-t^{m}}\frac{(z/w)^{m}}{m(-m)}\bigg)\varphi^{-}(p;z) \eta(p;w)\\
&=\exp\bigg(\sum_{m>0}\frac{(1-q^{m})(1-t^{-m})(1-\gamma^{2m})}{1-p^{m}}p^{m}\gamma^{m/2}\frac{(w/z)^{m}}{m}\bigg) \\
&\quad \times\exp\bigg(-\sum_{m>0}\frac{(1-q^{m})(1-t^{-m})(1-\gamma^{2m})}{1-p^{m}}\gamma^{-m/2}\frac{(z/w)^{m}}{m}\bigg)\varphi^{-}(p;z)\eta(p;w) \\
&=g_{p}\Big(\gamma^{-1/2}\frac{z}{w}\Big)^{-1}\varphi^{-}(p;z)\eta(p;w).
\end{align*}
Consequently we have $\varphi^{\pm}(p;z)\eta(p;w)=g_{p}\Big(\gamma^{\pm\frac{1}{2}}\displaystyle \frac{z}{w}\Big)\eta(p;w)\varphi^{\pm}(p;z)$.

Finally we show (3.25). Similar to the above calculations, we have the following :
\begin{align*}
&\quad \varphi^{+}(p;z)\xi(p;z) \\
&=\exp\bigg(-\sum_{m>0}\frac{1-t^{-m}}{1-p^{m}}p^{m}(1-\gamma^{-2m})\gamma^{m/2}\frac{1-t^{m}}{1-p^{m}}\gamma^{-m}p^{m}\cdot m\frac{1-p^{m}}{(qt^{-1}p)^{m}}\frac{1-q^{m}}{1-t^{m}}\frac{(z/w)^{m}}{m(-m)}\bigg) \\
&\quad\times\exp\bigg(-\sum_{m>0}\frac{1-t^{m}}{1-p^{m}}(1-\gamma^{2m})\gamma^{-m/2}\frac{1-t^{-m}}{1-p^{m}}\gamma^{m}\cdot m(1-p^{m})\frac{1-q^{m}}{1-t^{m}}\frac{(w/z)^{m}}{m(-m)}\bigg)\\
&\hskip 13cm \times\xi(p;w)\varphi^{+}(p;z) \\
&=\exp\bigg(-\sum_{m>0}\frac{(1-q^{m})(1-t^{-m})(1-\gamma^{2m})}{1-p^{m}}p^{m}\gamma^{-m/2}\frac{(z/w)^{m}}{m}\bigg) \\
&\quad \times\exp\bigg(\sum_{m>0}\frac{(1-q^{m})(1-t^{-m})(1-\gamma^{2m})}{1-p^{m}}\gamma^{m/2}\frac{(w/z)^{m}}{m}\bigg)\xi(p;w)\varphi^{+}(p;z) \\
&=g_{p}\Big(\gamma^{1/2}\frac{w}{z}\Big)\xi(p;w)\varphi^{+}(p;z)=g_{p}\Big(\gamma^{-1/2}\frac{z}{w}\Big)^{-1}\xi(p;w)\varphi^{+}(p;z),
\end{align*}
\begin{align*}
&\quad \xi(p;w)\varphi^{-}(p;z) \\
&=\exp\bigg(-\sum_{m>0}\frac{1-t^{-m}}{1-p^{m}}\gamma^{-m}p^{m}\frac{1-t^{m}}{1-p^{m}}p^{m}(1-\gamma^{-2m})\gamma^{m/2}\cdot m\frac{1-p^{m}}{(qt^{-1}p)^{m}}\frac{1-q^{m}}{1-t^{m}}\frac{(w/z)^{m}}{m(-m)}\bigg) \\
&\quad\times\exp\bigg(-\sum_{m>0}\frac{1-t^{m}}{1-p^{m}}\gamma^{m}\frac{1-t^{-m}}{1-p^{m}}(1-\gamma^{2m})\gamma^{-m/2}\cdot m(1-p^{m})\frac{1-q^{m}}{1-t^{m}}\frac{(z/w)^{m}}{m(-m)}\bigg)\\
&\hskip 13cm \times\varphi^{-}(p;z)\xi(p;w) \\
&=\exp\bigg(-\sum_{m>0}\frac{(1-q^{m})(1-t^{-m})(1-\gamma^{2m})}{1-p^{m}}p^{m}\gamma^{-m/2}\frac{(w/z)^{m}}{m}\bigg) \\
&\quad \times\exp\bigg(\sum_{m>0}\frac{(1-q^{m})(1-t^{-m})(1-\gamma^{2m})}{1-p^{m}}\gamma^{m/2}\frac{(z/w)^{m}}{m}\bigg)\varphi^{-}(p;z)\xi(p;w) \\
&=g_{p}\Big(\gamma^{1/2}\frac{z}{w}\Big)\varphi^{-}(p;z)\xi(p;w).
\end{align*}
Therefore we have $\varphi^{\pm}(p;z)\xi(p;w)=g_{p}\Big(\gamma^{\mp\frac{1}{2}}\displaystyle \frac{z}{w}\Big)^{-1}\xi(p;w)\varphi^{\pm}(p;z)$, here the proof of the theorem is complete. \quad $\Box$
\end{proof}

\subsection{Elliptic Ding-Iohara algebra $\mathcal{U}(q,t,p)$}
After the theorem 3.8, we can define the elliptic Ding-Iohara algebra.

\begin{definition}[\textbf{Elliptic Ding-Iohara algebra $\mathcal{U}(q,t,p)$}] 
Let $g_{p}(x)$ be the structure function defined by (3.22) :
\begin{align*}
g_{p}(x)=\frac{\Theta_{p}(qx)\Theta_{p}(t^{-1}x)\Theta_{p}(q^{-1}tx)}{\Theta_{p}(q^{-1}x)\Theta_{p}(tx)\Theta_{p}(qt^{-1}x)}.
\end{align*}
Let $\gamma$ be the central, invertible element and currents $x^{\pm}(p;z):=\sum_{n\in\mathbf{Z}}x^{\pm}_{n}(p)z^{-n}$, $\psi^{\pm}(p;z):=\sum_{n\in\mathbf{Z}}\psi^{\pm}_{n}(p)z^{-n}$ be operators subject to the defining relations listed below.
\begin{align}
&\hskip 1cm [\psi^{\pm}(p;z), \psi^{\pm}(p;w)]=0, \quad \psi^{+}(p;z)\psi^{-}(p;w)=\frac{g_{p}(\gamma z/w)}{g_{p}(\gamma^{-1}z/w)}\psi^{-}(p;w)\psi^{+}(p;z),\notag\\
&\hskip 3cm \psi^{\pm}(p;z)x^{+}(p;w)=g_{p}\Big(\gamma^{\pm\frac{1}{2}}\frac{z}{w}\Big)x^{+}(p;w)\psi^{\pm}(p;z),\notag\\
&\hskip 3cm \psi^{\pm}(p;z)x^{-}(p;w)=g_{p}\Big(\gamma^{\mp\frac{1}{2}}\frac{z}{w}\Big)^{-1}x^{-}(p;w)\psi^{\pm}(p;z),\notag\\
&\hskip 3cm x^{\pm}(p;z)x^{\pm}(p;w)=g_{p}\Big(\frac{z}{w}\Big)^{\pm 1}x^{\pm}(p;w)x^{\pm}(p;z),\notag\\
&[x^{+}(p;z),x^{-}(p;w)]
=\frac{\Theta_{p}(q)\Theta_{p}(t^{-1})}{(p;p)_{\infty}^{3}\Theta_{p}(qt^{-1})}\bigg\{\delta\Big(\gamma\frac{w}{z}\Big)\psi^{+}(p;\gamma^{1/2}w)-\delta\Big(\gamma^{-1}\frac{w}{z}\Big)\psi^{-}(p;\gamma^{-1/2}w)\bigg\}.
\end{align}
Then we define the elliptic Ding-Iohara algebra $\mathcal{U}(q,t,p)$ to be the associative $\mathbf{C}$-algebra generated by $\{x^{\pm}_{n}(p)\}_{n\in{\mathbf{Z}}}$, $\{\psi^{\pm}_{n}(p)\}_{n\in{\mathbf{Z}}}$ and $\gamma$.
\end{definition}

Similar to the trigonometric case, the map defined by
\begin{align*}
\gamma \mapsto (qt^{-1})^{-1/2}, \quad x^{+}(p;z) \mapsto \eta(p;z), \quad x^{-}(p;z) \mapsto \xi(p;z), \quad \psi^{\pm}(p;z) \mapsto \varphi^{\pm}(p;z)
\end{align*}
gives a representation, or the free field realization, of the elliptic Ding-Iohara algebra $\mathcal{U}(q,t,p)$ (Theorem 1.2 in section 1).

\begin{remark}
(1) By the definition, the trigonometric limit $p \to 0$ of the elliptic Ding-Iohara algebra $\mathcal{U}(q,t,p)$ degenerates to the Ding-Iohara algebra $\mathcal{U}(q,t)$.

(2) Since the relations (3.32) take the same forms of the trigonometric case (2.18), we can define the coproduct $\Delta : \mathcal{U}(q,t,p) \to \mathcal{U}(q,t,p)\otimes\mathcal{U}(q,t,p)$ similar to the trigonometric case :
\begin{align}
&\Delta(\gamma^{\pm 1})=\gamma^{\pm 1}\otimes\gamma^{\pm 1}, \quad \Delta(\psi^{\pm}(p;z))=\psi^{\pm}(p;\gamma^{\pm1/2}_{(2)}z)\otimes\psi^{\pm}(p;\gamma^{\mp1/2}_{(1)}z), \notag\\
&\Delta(x^{+}(p;z))=x^{+}(p;z)\otimes 1+\psi^{-}(p;\gamma^{1/2}_{(1)}z)\otimes x^{+}(p;\gamma_{(1)}z), \notag\\
&\Delta(x^{-}(p;z))=x^{-}(p;\gamma_{(2)}z)\otimes \psi^{+}(p;\gamma^{1/2}_{(2)}z)+1\otimes x^{-}(p;z).
\end{align}

(3) In [17], another elliptic Ding-Iohara algebra is defined based on the idea of the quasi-Hopf deformation. Then the same structure function defined by (3.22) arises.
\end{remark}

\section{Free field realization of the elliptic Macdonald operator}
In this section, we study the relations between the elliptic currents $\eta(p;z)$, $\xi(p;z)$ and the elliptic Macdonald operators $H_{N}(q,t,p)$, $H_{N}(q^{-1},t^{-1},p)$.

\subsection{Preparations}
The elliptic Macdonald operator $H_{N}(q,t,p)$ $(N\in{\mathbf{Z}_{>0}})$ is defined as follows.
\begin{align*}
H_{N}(q,t,p):=\sum_{i=1}^{N}\prod_{j\neq{i}}\frac{\Theta_{p}(tx_{i}/x_{j})}{\Theta_{p}(x_{i}/x_{j})}T_{q,x_{i}}.
\end{align*}
First, we need a lemma to calculate the constant term of a product of the theta functions.

\begin{lemma}
(1) We have a partial fraction expansion formula of the product of the theta functions as follows.
\begin{align}
\prod_{i=1}^{N}\frac{\Theta_{p}(t^{-1}x_{i}z)}{\Theta_{p}(x_{i}z)}=\frac{\Theta_{p}(t)}{\Theta_{p}(t^{N})}\sum_{i=1}^{N}\frac{\Theta_{p}(t^{-N}x_{i}z)}{\Theta_{p}(x_{i}z)}\prod_{j\neq{i}}\frac{\Theta_{p}(tx_{i}/x_{j})}{\Theta_{p}(x_{i}/x_{j})}.
\end{align}
(2) From Ramanujan's summation formula :
\begin{align}
\sum_{n\in{\mathbf{Z}}}\frac{(a;p)_{n}}{(b;p)_{n}}z^{n}=\frac{(az;p)_{\infty}(p/az;p)_{\infty}(b/a;p)_{\infty}(p;p)_{\infty}}{(z;p)_{\infty}(b/az;p)_{\infty}(p/a;p)_{\infty}(b;p)_{\infty}} \quad (|a^{-1}b|<|z|<1),
\end{align}
\textit{we have an expansion of $\Theta_{p}(az)/\Theta_{p}(z)$ as follows :}
\begin{align}
\frac{\Theta_{p}(az)}{\Theta_{p}(z)}=\frac{\Theta_{p}(a)}{(p;p)_{\infty}^{3}}\sum_{n\in{\mathbf{Z}}}\frac{z^{n}}{1-ap^{n}} \quad (|p|<|z|<1).
\end{align}
\end{lemma}

\begin{proof}[\textit{Proof}]
(1) In the partial fraction expansion (6.19) in Appendix B as
\begin{align*}
\prod_{i=1}^{N}\frac{\Theta_{p}(t_{i}x_{i}^{-1}z)}{\Theta_{p}(x_{i}^{-1}z)}=\sum_{i=1}^{N}\frac{\Theta_{p}(t_{i})}{\Theta_{p}(t_{(N)})}\frac{\Theta_{p}(t_{(N)}x_{i}^{-1}z)}{\Theta_{p}(x_{i}^{-1}z)}\prod_{j\neq{i}}\frac{\Theta_{p}(t_{j}x_{i}/x_{j})}{\Theta_{p}(x_{i}/x_{j})},
\end{align*}
by setting $t_{j}=t^{-1}$ and substitution $x_{j} \to x_{j}^{-1}$, we obtain (4.1).

(2) In Ramanujan's summation formula (4.2) (the proof of this formula is in Appendix B), by setting $b/a=p$ we have
\begin{align*}
&\text{(Left hand side of (4.2))}=\sum_{n\in{\mathbf{Z}}}\frac{(a;p)_{n}}{(ap;p)_{n}}z^{n}=(1-a)\sum_{n\in{\mathbf{Z}}}\frac{z^{n}}{1-ap^{n}}, \\
&\text{(Right hand side of (4.2))}=\frac{(az;p)_{\infty}(p(az)^{-1};p)_{\infty}(p;p)_{\infty}^{2}}{(z;p)_{\infty}(pz;p)_{\infty}(pa^{-1};p)_{\infty}(pa;p)_{\infty}}=\frac{(1-a)(p;p)_{\infty}^{3}}{\Theta_{p}(a)}\frac{\Theta_{p}(az)}{\Theta_{p}(z)}. \quad \Box
\end{align*}
\end{proof}

\begin{remark}
Using the equation (6.19) and the relation between the theta function and the delta function, we have the following :
\begin{align*}
\prod_{i=1}^{N}\frac{\Theta_{p}(t_{i}x_{i}^{-1}z)}{\Theta_{p}(x_{i}^{-1}z)}
&=\sum_{i=1}^{N}\frac{\Theta_{p}(t_{i})}{(p;p)_{\infty}^{3}}\prod_{j\neq{i}}\frac{\Theta_{p}(t_{j}x_{i}/x_{j})}{\Theta_{p}(x_{i}/x_{j})}\delta(x_{i}^{-1}z)+t_{(N)}\prod_{i=1}^{N}\frac{\Theta_{p}(t_{i}^{-1}x_{i}/z)}{\Theta_{p}(x_{i}/z)} \\
&=\sum_{i=1}^{N}\frac{\Theta_{p}(t_{i})}{(p;p)_{\infty}^{3}}\prod_{j\neq{i}}\frac{\Theta_{p}(t_{j}x_{i}/x_{j})}{\Theta_{p}(x_{i}/x_{j})}\delta(x_{i}^{-1}z)+t_{(N)}\prod_{i=1}^{N}\frac{\Theta_{p}(pt_{i}x_{i}^{-1}z)}{\Theta_{p}(px_{i}^{-1}z)},
\end{align*}
where we use the relation $1/\Theta_{p}(px)=1/\Theta_{p}(x^{-1})$. Taking the constant term of the above relation in $z$, we have the identity as follows.
\begin{align}
\sum_{i=1}^{N}\Theta_{p}(t_{i})\prod_{j\neq{i}}\frac{\Theta_{p}(t_{j}x_{i}/x_{j})}{\Theta_{p}(x_{i}/x_{j})}=(1-t_{(N)})(p;p)_{\infty}^{3}\left[\prod_{i=1}^{N}\frac{\Theta_{p}(t_{i}x_{i}^{-1}z)}{\Theta_{p}(x_{i}^{-1}z)}\right]_{1}.
\end{align}
\end{remark}

\subsection{The case of using $\eta(p;z)$}

\begin{theorem}[\textbf{Free field realization of the elliptic Macdonald operator}] 
We use the notation $\phi_{N}(p;x):=\prod_{j=1}^{N}\phi(p;x_{j})$ $(N\in\mathbf{Z}_{>0})$. Then the elliptic current $\eta(p;z)$ reproduces the elliptic Macdonald operator $H_{N}(q,t,p)$ as follows :
\begin{align}
[\eta(p;z)-t^{-N}(\eta(p;z))_{-}(\eta(p;p^{-1}z))_{+}]_{1}\phi_{N}(p;x)|0 \rangle 
=\frac{t^{-N+1}\Theta_{p}(t^{-1})}{(p;p)_{\infty}^{3}}H_{N}(q,t,p)\phi_{N}(p;x)|0 \rangle.
\end{align}
\end{theorem}

\begin{proof}[\textit{Proof}]
First, we show the relation 
\begin{align}
\eta(p;z)\phi(p;w)=\frac{\Theta_{p}(w/z)}{\Theta_{p}(tw/z)}\bm{:}\eta(p;z)\phi(p;w)\bm{:}.
\end{align}
This is shown by Wick's theorem :
\begin{align*}
&\quad \eta(p;z)\phi(p;w) \\
&=\exp\bigg(-\sum_{m>0}\frac{1-t^{-m}}{1-p^{m}}p^{m}\frac{(1-t^{m})(qt^{-1}p)^{m}}{(1-q^{m})(1-p^{m})}\cdot m\frac{1-p^{m}}{(qt^{-1}p)^{m}}\frac{1-q^{m}}{1-t^{m}}\frac{(z/w)^{m}}{m\cdot m}\bigg) \\
&\times\exp\bigg(-\sum_{m>0}\frac{1-t^{m}}{1-p^{m}}\frac{1-t^{m}}{(1-q^{m})(1-p^{m})}\cdot m(1-p^{m})\frac{1-q^{m}}{1-t^{m}}\frac{(w/z)^{m}}{m\cdot m}\bigg)\bm{:}\eta(p;z)\phi(p;w)\bm{:} \\
&=\exp\bigg(-\sum_{m>0}\frac{1-t^{-m}}{1-p^{m}}p^{m}\frac{(z/w)^{m}}{m}\bigg)\exp\bigg(-\sum_{m>0}\frac{1-t^{m}}{1-p^{m}}\frac{(w/z)^{m}}{m}\bigg)\bm{:}\eta(p;z)\phi(p;w)\bm{:} \\
&=\frac{(pz/w;p)_{\infty}}{(t^{-1}pz/w;p)_{\infty}}\frac{(w/z;p)_{\infty}}{(tw/z;p)_{\infty}}\bm{:}\eta(p;z)\phi(p;w)\bm{:} \\
&=\frac{\Theta_{p}(w/z)}{\Theta_{p}(tw/z)}\bm{:}\eta(p;z)\phi(p;w)\bm{:}.
\end{align*}
By this relation, we have
\begin{align}
\eta(p;z)\phi_{N}(p;x)=\prod_{i=1}^{N}\frac{\Theta_{p}(x_{i}/z)}{\Theta_{p}(tx_{i}/z)}\bm{:}\eta(p;z)\phi_{N}(p;x)\bm{:}.
\end{align}
Using (3.17) and (4.1), we can check the following relation
\begin{align*}
\prod_{i=1}^{N}\frac{\Theta_{p}(x_{i}/z)}{\Theta_{p}(tx_{i}/z)}
=\frac{t^{-N+1}\Theta_{p}(t^{-1})}{(p;p)_{\infty}^{3}}\sum_{i=1}^{N}\prod_{j\neq{i}}\frac{\Theta_{p}(tx_{i}/x_{j})}{\Theta_{p}(x_{i}/x_{j})}\delta\Big(t\frac{x_{i}}{z}\Big)+t^{-N}\prod_{i=1}^{N}\frac{\Theta_{p}(z/x_{i})}{\Theta_{p}(t^{-1}z/x_{i})}.
\end{align*}
By these relations we have the following :
\begin{align}
&\quad \eta(p;z)\phi_{N}(p;x)|0 \rangle 
=\prod_{i=1}^{N}\frac{\Theta_{p}(x_{i}/z)}{\Theta_{p}(tx_{i}/z)}(\eta(p;z))_{-}\phi_{N}(p;x)|0 \rangle \notag\\
&=\frac{t^{-N+1}\Theta_{p}(t^{-1})}{(p;p)_{\infty}^{3}}\sum_{i=1}^{N}\prod_{j\neq{i}}\frac{\Theta_{p}(tx_{i}/x_{j})}{\Theta_{p}(x_{i}/x_{j})}\delta\Big(t\frac{x_{i}}{z}\Big)(\eta(p;tx_{i}))_{-}\phi_{N}(p;x)|0 \rangle \notag\\
&\hskip 5cm +t^{-N}\prod_{i=1}^{N}\frac{\Theta_{p}(z/x_{i})}{\Theta_{p}(t^{-1}z/x_{i})}(\eta(p;z))_{-}\phi_{N}(p;x)|0 \rangle \notag\\
&=\frac{t^{-N+1}\Theta_{p}(t^{-1})}{(p;p)_{\infty}^{3}}\sum_{i=1}^{N}\prod_{j\neq{i}}\frac{\Theta_{p}(tx_{i}/x_{j})}{\Theta_{p}(x_{i}/x_{j})}\delta\Big(t\frac{x_{i}}{z}\Big)T_{q,x_{i}}\phi_{N}(p;x)|0 \rangle \notag\\
&\hskip 5cm +t^{-N}\prod_{i=1}^{N}\frac{\Theta_{p}(z/x_{i})}{\Theta_{p}(t^{-1}z/x_{i})}(\eta(p;z))_{-}\phi_{N}(p;x)|0 \rangle,
\end{align}
where we use the relation $(\eta(p;tz))_{-}\phi(p;z)=\phi(p;qz)=T_{q,z}\phi(p;z)$. Let us recall the relation $1/\Theta_{p}(px)=1/\Theta_{p}(x^{-1})$. This leads that
\begin{align*}
\prod_{i=1}^{N}\frac{\Theta_{p}(z/x_{i})}{\Theta_{p}(t^{-1}z/x_{i})}=\prod_{i=1}^{N}\frac{\Theta_{p}(px_{i}/z)}{\Theta_{p}(ptx_{i}/z)}.
\end{align*}
Hence we have the following.
\begin{align}
\prod_{i=1}^{N}\frac{\Theta_{p}(z/x_{i})}{\Theta_{p}(t^{-1}z/x_{i})}(\eta(p;z))_{-}\phi_{N}(p;x)|0 \rangle 
=(\eta(p;z))_{-}(\eta(p;p^{-1}z))_{+}\phi_{N}(p;x)|0 \rangle.
\end{align}
Finally we have
\begin{align*}
[\eta(p;z)-t^{-N}(\eta(p;z))_{-}(\eta(p;p^{-1}z))_{+}]_{1}\phi_{N}(p;x)|0 \rangle 
=\frac{t^{-N+1}\Theta_{p}(t^{-1})}{(p;p)_{\infty}^{3}}H_{N}(q,t,p)\phi_{N}(p;x)|0 \rangle. \quad \Box
\end{align*}
\end{proof}

\begin{remark}
Let us define $C_{N}(p;x,y)$ by
\begin{align*}
C_{N}(p;x,y):=\langle 0|\phi^{\ast}_{N}(p;x)[(\eta(p;z))_{-}(\eta(p;p^{-1}z))_{+}]_{1}\phi_{N}(p;x)|0 \rangle \Big/\Pi(q,t,p)(x,y). 
\end{align*}
By Wick's theorem, we have
\begin{align}
C_{N}(p;x,y)=\left[\prod_{i=1}^{N}\frac{\Theta_{p}(t^{-1}x_{i}z)\Theta_{p}(z/y_{i})}{\Theta_{p}(x_{i}z)\Theta_{p}(t^{-1}z/y_{i})}\right]_{1}.
\end{align}
By the relation (4.3) which is obtained from Ramanujan's summation formula, the explicit form of $C_{N}(p;x,y)$ takes the following :
\begin{align}
C_{N}(p;x,y)=\bigg(\frac{t^{-N+1}\Theta_{p}(t^{-1})}{(p;p)_{\infty}^{3}}\bigg)^{2}\sum_{\genfrac{}{}{0pt}{1}{1\leq{i}\leq{N}}{1\leq{k}\leq{N}}}\prod_{j\neq{i}}\frac{\Theta_{p}(tx_{i}/x_{j})}{\Theta_{p}(x_{i}/x_{j})}\prod_{\ell\neq{k}}\frac{\Theta_{p}(ty_{k}/y_{\ell})}{\Theta_{p}(y_{k}/y_{\ell})}\sum_{m\in\mathbf{Z}}\frac{(tpx_{i}y_{k})^{m}}{(1-t^{-N}p^{m})^{2}}.
\end{align}
In the trigonometric limit, $C_{N}(p;x,y)$ degenerates to $1$ : $C_{N}(p;x,y) \xrightarrow[p \to 0]{} 1$.
\end{remark}

\subsection{The case of using $\xi(p;z)$}
Instead of using $\eta(p;z)$, we can carry out similar calculations by using $\xi(p;z)$. Then we have the following theorem. 

\begin{theorem}[\textbf{Free field realization of the elliptic Macdonald operator by $\xi(p;z)$}] 
The elliptic current $\xi(p;z)$ reproduces the elliptic Macdonald operator $H_{N}(q^{-1},t^{-1},p)$ as follows :
\begin{align}
[\xi(p;z)-t^{N}(\xi(p;z))_{-}(\xi(p;p^{-1}z))_{+}]_{1}\phi_{N}(p;x)|0 \rangle 
=\frac{t^{N-1}\Theta_{p}(t)}{(p;p)_{\infty}^{3}}H_{N}(q^{-1},t^{-1},p)\phi_{N}(p;x)|0 \rangle.
\end{align}
\end{theorem}

The proof of the theorem 4.5 is similar to the theorem 4.3.

\subsection{Another forms of the theorem 4.3, 4.5}
Let us introduce the zero mode generators $a_{0},\,Q$ satisfying the following :
\begin{align}
 [a_{0},Q]=1,\quad [a_{n},a_{0}]=[b_{n},a_{0}]=0, \quad [a_{n},Q]=[b_{n},Q]=0 \quad (n\in\mathbf{Z}\setminus\{0\}).
\end{align}
For a complex number $\alpha$, we define $|\alpha \rangle:=e^{\alpha Q}|0 \rangle$. Then we have $a_{0}|\alpha \rangle=\alpha|\alpha \rangle$.

By using the zero modes, we can reformulate the free field realization of the elliptic Macdonald operator as follows.

\begin{theorem}
Set $\widetilde{\eta}(p;z):=(\eta(p;z))_{-}(\eta(p;p^{-1}z))_{+}$, $\widetilde{\xi}(p;z):=(\xi(p;z))_{-}(\xi(p;p^{-1}z))_{+}$. We define operators $E(p;z)$, $F(p;z)$ as follows :
\begin{align}
E(p;z):=\eta(p;z)-\widetilde{\eta}(p;z)t^{-a_{0}}, \quad F(p;z):=\xi(p;z)-\widetilde{\xi}(p;z)t^{a_{0}}.
\end{align}
Then the elliptic Macdonald operators $H_{N}(q,t,p)$, $H_{N}(q^{-1},t^{-1},p)$ are reproduced from the operators $E(p;z)$, $F(p;z)$ as follows :
\begin{align}
&[E(p;z)]_{1}\phi_{N}(p;x)|N \rangle=\frac{t^{-N+1}\Theta_{p}(t^{-1})}{(p;p)_{\infty}^{3}}H_{N}(q,t,p)\phi_{N}(p;x)|N \rangle, \\
&[F(p;z)]_{1}\phi_{N}(p;x)|N \rangle=\frac{t^{N-1}\Theta_{p}(t)}{(p;p)_{\infty}^{3}}H_{N}(q^{-1},t^{-1},p)\phi_{N}(p;x)|N \rangle.
\end{align}
\end{theorem}

\section{Some observations and remarks}
To end this paper, we indicate what remains unclear or should be clarified and give some comments on concerned materials.

\subsection{The method of elliptic deformation}
Looking at the construction of the elliptic currents such as $\eta(p;z), \, \xi(p;z)$ again, we can define a procedure of the elliptic deformation as follows.

\begin{definition}[\textbf{The method of elliptic deformation}] 
Suppose $X(z)$ be an operator of the form
\begin{align*}
X(z)=\exp\bigg(\sum_{n<0}X^{-}_{n}a_{n}z^{-n}\bigg)\exp\bigg(\sum_{n>0}X^{+}_{n}a_{n}z^{-n}\bigg) \quad(X^{\pm}_{n}\in\mathbf{C}),
\end{align*}
where $\{a_{n}\}_{n\in{\mathbf{Z}\setminus\{0\}}}$ are boson generators which satisfy the relation :
\begin{align*}
[a_{m},a_{n}]=m\frac{1-q^{|m|}}{1-t^{|m|}}\delta_{m+n,0}.
\end{align*}
Then the method of elliptic deformation is a procedure as follows :

(Step 1). Change boson generators into the ones satisfying the relations :
\begin{align*}
&[a_{m},a_{n}]=m(1-p^{|m|})\frac{1-q^{|m|}}{1-t^{|m|}}\delta_{m+n,0}, \quad [b_{m},b_{n}]=m\frac{1-p^{|m|}}{(qt^{-1}p)^{|m|}}\frac{1-q^{|m|}}{1-t^{|m|}}\delta_{m+n,0}, \notag\\
&[a_{m},b_{n}]=0.
\end{align*}

(Step 2). Set $X(p;z):=X_{b}(p;z)X_{a}(p;z)$, where
\begin{align}
&X_{b}(p;z):=\exp\bigg(-\sum_{n<0}\frac{p^{|n|}}{1-p^{|n|}}X^{-}_{-n}b_{n}z^{n}\bigg)\exp\bigg(-\sum_{n>0}\frac{p^{|n|}}{1-p^{|n|}}X^{+}_{-n}b_{n}z^{n}\bigg), \\
&X_{a}(p;z):=\exp\bigg(\sum_{n<0}\frac{1}{1-p^{|n|}}X^{-}_{n}a_{n}z^{-n}\bigg)\exp\bigg(\sum_{n>0}\frac{1}{1-p^{|n|}}X^{+}_{n}a_{n}z^{-n}\bigg).
\end{align}
\end{definition}

\subsection{Commutator of operators $E(p;z)$, $F(p;z)$}
In section 3, we showed the proposition 3.7 as
\begin{align*}
[\eta(p;z),\xi(p;w)]=\frac{\Theta_{p}(q)\Theta_{p}(t^{-1})}{(p;p)_{\infty}^{3}\Theta_{p}(qt^{-1})}\bigg\{\delta\Big(\gamma\frac{w}{z}\Big)\varphi^{+}(p;\gamma^{1/2}w)-\delta\Big(\gamma^{-1}\frac{w}{z}\Big)\varphi^{-}(p;\gamma^{-1/2}w)\bigg\}.
\end{align*} 
Since $[\varphi^{+}(p;z)]_{1}\neq [\varphi^{-}(p;z)]_{1}$, we have $[[\eta(p;z)]_{1},[\xi(p;w)]_{1}]\neq 0$. Compared to this, operators $E(p;z)$, $F(p;z)$ defined in (4.14) satisfy the following.

\begin{prop}
(1) For operators $E(p;z)$, $F(p;z)$ we have
\begin{align}
E(p;z)E(p;w)&=g_{p}\Big(\frac{z}{w}\Big)E(p;w)E(p;z), \\
F(p;z)F(p;w)&=g_{p}\Big(\frac{z}{w}\Big)^{-1}F(p;w)F(p;z).
\end{align}
(2) The commutator of operators $E(p;z)$, $F(p;z)$ takes the form as
\begin{align}
[E(p;z),F(p;w)]=\frac{\Theta_{p}(q)\Theta_{p}(t^{-1})}{(p;p)_{\infty}^{3}\Theta_{p}(qt^{-1})}\delta\Big(\gamma\frac{w}{z}\Big)\{\varphi^{+}(p;\gamma^{1/2}w)-\varphi^{+}(p;\gamma^{1/2}p^{-1}w)\}.
\end{align}
\end{prop} 

\begin{proof}
(1) Here we are going to show (5.3). First we can check the following :
\begin{align}
\eta(p;z)\widetilde{\eta}(p;w)&=\frac{\Theta_{p}(w/z)\Theta_{p}(qt^{-1}w/z)}{\Theta_{p}(qw/z)\Theta_{p}(t^{-1}w/z)}\bm{:}\eta(p;z)\widetilde{\eta}(p;w)\bm{:}, \\
\widetilde{\eta}(p;z)\eta(p;w)&=\frac{\Theta_{p}(w/p^{-1}z)\Theta_{p}(qt^{-1}w/p^{-1}z)}{\Theta_{p}(qw/p^{-1}z)\Theta_{p}(t^{-1}w/p^{-1}z)}\bm{:}\widetilde{\eta}(p;z)\widetilde{\eta}(p;w)\bm{:} \notag\\
&=\frac{\Theta_{p}(z/w)\Theta_{p}(q^{-1}tz/w)}{\Theta_{p}(q^{-1}z/w)\Theta_{p}(tz/w)}\bm{:}\widetilde{\eta}(p;z)\eta(p;w)\bm{:}, \\
\widetilde{\eta}(p;z)\widetilde{\eta}(p;w)&=\frac{\Theta_{p}(z/w)\Theta_{p}(q^{-1}tz/w)}{\Theta_{p}(q^{-1}z/w)\Theta_{p}(tz/w)}\bm{:}\widetilde{\eta}(p;z)\widetilde{\eta}(p;w)\bm{:}.
\end{align}
From them we have
\begin{align}
\eta(p;z)\widetilde{\eta}(p;w)&=g_{p}\Big(\frac{z}{w}\Big)\widetilde{\eta}(p;w)\eta(p;z), \\
\widetilde{\eta}(p;z)\eta(p;w)&=g_{p}\Big(\frac{z}{w}\Big)\eta(p;w)\widetilde{\eta}(p;z), \\
\widetilde{\eta}(p;z)\widetilde{\eta}(p;w)&=g_{p}\Big(\frac{z}{w}\Big)\widetilde{\eta}(p;w)\widetilde{\eta}(p;z).
\end{align}
Hence we have the following.
\begin{align*}
&\quad E(p;z)E(p;w)=(\eta(p;z)-\widetilde{\eta}(p;z)t^{-a_{0}})(\eta(p;w)-\widetilde{\eta}(p;w)t^{-a_{0}}) \\
&=\eta(p;z)\eta(p;w)-\eta(p;z)\widetilde{\eta}(p;w)t^{-a_{0}}-\widetilde{\eta}(p;z)\eta(p;w)t^{-a_{0}}+\widetilde{\eta}(p;z)\widetilde{\eta}(p;w)t^{-2a_{0}} \\
&=g_{p}\Big(\frac{z}{w}\Big)(\eta(p;w)\eta(p;z)-\eta(p;w)\widetilde{\eta}(p;z)t^{-a_{0}}-\widetilde{\eta}(p;w)\eta(p;z)t^{-a_{0}}+\widetilde{\eta}(p;w)\widetilde{\eta}(p;z)t^{-2a_{0}}) \\
&=g_{p}\Big(\frac{z}{w}\Big)E(p;w)E(p;z).
\end{align*}

(2) Let us recall the relations shown in the proposition 3.7 as 
\begin{align}
&\eta(p;z)\xi(p;w)=\frac{\Theta_{p}(q\gamma w/z)\Theta_{p}(q^{-1}\gamma^{-1}w/z)}{\Theta_{p}(\gamma w/z)\Theta_{p}(\gamma^{-1}w/z)}\bm{:}\eta(p;z)\xi(p;w)\bm{:}, \\
&\xi(p;w)\eta(p;z)=\frac{\Theta_{p}(q\gamma z/w)\Theta_{p}(q^{-1}\gamma^{-1}z/w)}{\Theta_{p}(\gamma z/w)\Theta_{p}(\gamma^{-1}z/w)}\bm{:}\eta(p;z)\xi(p;w)\bm{:}.
\end{align}
We define $A(x)$ as follows :
\begin{align}
A(x):=\frac{\Theta_{p}(q\gamma x)\Theta_{p}(q^{-1}\gamma^{-1}x)}{\Theta_{p}(\gamma x)\Theta_{p}(\gamma^{-1}x)}.
\end{align}
Then we have
\begin{align*}
&\quad E(p;z)F(p;w)=(\eta(p;z)-\widetilde{\eta}(p;z)t^{-a_{0}})(\xi(p;w)-\widetilde{\xi}(p;w)t^{a_{0}}) \\
&=A\Big(\frac{w}{z}\Big)\bm{:}\eta(p;z)\xi(p;w)\bm{:}-A\Big(\frac{w}{z}\Big)\bm{:}\eta(p;z)\widetilde{\xi}(p;w)\bm{:}t^{a_{0}} \\
&\hskip 2cm -A\Big(\frac{w}{p^{-1}z}\Big)\bm{:}\widetilde{\eta}(p;z)\xi(p;w)\bm{:}t^{-a_{0}}+A\Big(\frac{w}{p^{-1}z}\Big)\bm{:}\widetilde{\eta}(p;z)\widetilde{\xi}(p;w)\bm{:} \\
&=A\Big(\frac{w}{z}\Big)\bm{:}\eta(p;z)\xi(p;w)\bm{:}-A\Big(\frac{w}{z}\Big)\bm{:}\eta(p;z)\widetilde{\xi}(p;w)\bm{:}t^{a_{0}} \\
&\hskip 2cm -A\Big(\frac{z}{w}\Big)\bm{:}\widetilde{\eta}(p;z)\xi(p;w)\bm{:}t^{-a_{0}}+A\Big(\frac{z}{w}\Big)\bm{:}\widetilde{\eta}(p;z)\widetilde{\xi}(p;w)\bm{:},
\end{align*}
\begin{align*}
&\quad F(p;w)E(p;z)=(\xi(p;w)-\widetilde{\xi}(p;w)t^{a_{0}})(\eta(p;z)-\widetilde{\eta}(p;z)t^{-a_{0}}) \\
&=A\Big(\frac{z}{w}\Big)\bm{:}\eta(p;z)\xi(p;w)\bm{:}-A\Big(\frac{z}{w}\Big)\bm{:}\widetilde{\eta}(p;z)\xi(p;w)\bm{:}t^{-a_{0}} \\
&\hskip 2cm -A\Big(\frac{z}{p^{-1}w}\Big)\bm{:}\eta(p;z)\widetilde{\xi}(p;w)\bm{:}t^{a_{0}}+A\Big(\frac{z}{p^{-1}w}\Big)\bm{:}\widetilde{\eta}(p;z)\widetilde{\xi}(p;w)\bm{:} \\
&=A\Big(\frac{z}{w}\Big)\bm{:}\eta(p;z)\xi(p;w)\bm{:}-A\Big(\frac{z}{w}\Big)\bm{:}\widetilde{\eta}(p;z)\xi(p;w)\bm{:}t^{-a_{0}} \\
&\hskip 2cm -A\Big(\frac{w}{z}\Big)\bm{:}\eta(p;z)\widetilde{\xi}(p;w)\bm{:}t^{a_{0}}+A\Big(\frac{w}{z}\Big)\bm{:}\widetilde{\eta}(p;z)\widetilde{\xi}(p;w)\bm{:}. 
\end{align*}
Here we use the relation $A(px)=A(x^{-1})$. From them we have
\begin{align*}
[E(p;z),F(p;w)]=\bigg\{A\Big(\frac{w}{z}\Big)-A\Big(\frac{z}{w}\Big)\bigg\}(\bm{:}\eta(p;z)\xi(p;w)\bm{:}-\bm{:}\widetilde{\eta}(p;z)\widetilde{\xi}(p;w)\bm{:}).
\end{align*}
Let us recall that
\begin{align}
A(x)-A(x^{-1})=\frac{\Theta_{p}(q)\Theta_{p}(t^{-1})}{(p;p)_{\infty}^{3}\Theta_{p}(qt^{-1})}\{\delta(\gamma x)-\delta(\gamma^{-1}x)\}.
\end{align}
Using the relation we have
\begin{align*}
&\quad [E(p;z),F(p;w)] \\
&=\frac{\Theta_{p}(q)\Theta_{p}(t^{-1})}{(p;p)_{\infty}^{3}\Theta_{p}(qt^{-1})}\bigg\{\delta\Big(\gamma\frac{w}{z}\Big)-\delta\Big(\gamma^{-1}\frac{w}{z}\Big)\bigg\}(\bm{:}\eta(p;z)\xi(p;w)\bm{:}-\bm{:}\widetilde{\eta}(p;z)\widetilde{\xi}(p;w)\bm{:}) \\
&=\frac{\Theta_{p}(q)\Theta_{p}(t^{-1})}{(p;p)_{\infty}^{3}\Theta_{p}(qt^{-1})}\delta\Big(\gamma\frac{w}{z}\Big)(\bm{:}\eta(p;\gamma w)\xi(p;w)\bm{:}-\bm{:}\widetilde{\eta}(p;\gamma w)\widetilde{\xi}(p;w)\bm{:}) \\
&\quad -\frac{\Theta_{p}(q)\Theta_{p}(t^{-1})}{(p;p)_{\infty}^{3}\Theta_{p}(qt^{-1})}\delta\Big(\gamma^{-1}\frac{w}{z}\Big)(\bm{:}\eta(p;\gamma^{-1}w)\xi(p;w)\bm{:}-\bm{:}\widetilde{\eta}(p;\gamma^{-1}w)\widetilde{\xi}(p;w)\bm{:}).
\end{align*}
Then we have $\bm{:}\eta(p;\gamma w)\xi(p;w)\bm{:}=\varphi^{+}(p;\gamma^{1/2}w)$, $\bm{:}\eta(p;\gamma^{-1}w)\xi(p;w)\bm{:}=\varphi^{-}(p;\gamma^{-1/2}w)$ and also have
\begin{align*}
\bm{:}\widetilde{\eta}(p;\gamma w)\widetilde{\xi}(p;w)\bm{:}
&=(\eta(p;\gamma w))_{-}(\xi(p;w))_{-}(\eta(p;\gamma p^{-1}w))_{+}(\xi(p; p^{-1}w))_{+} \\
&=\varphi^{+}(p;\gamma^{1/2}p^{-1}w), \\
\bm{:}\widetilde{\eta}(p;\gamma^{-1}w)\widetilde{\xi}(p;w)\bm{:}
&=(\eta(p;\gamma^{-1}w))_{-}(\xi(p;w))_{-}(\eta(p;\gamma^{-1}p^{-1}w))_{+}(\xi(p;p^{-1}w))_{+} \\
&=\varphi^{-}(p;\gamma^{-1/2}w).
\end{align*}
Therefore we have (5.5). \quad $\Box$
\end{proof}

\begin{remark}
From the relation (5.5), we have the commutativity of constant terms $[E(p;z)]_{1}$, $[F(p;z)]_{1}$ : $[[E(p;z)]_{1},[F(p;w)]_{1}]=0$. This corresponds to the commutativity of the elliptic Macdonald operators as $[H_{N}(q,t,p),H_{N}(q^{-1},t^{-1},p)]=0$. It seems that for the free field realization of the elliptic Macdonald operator, we should use the operators $E(p;z)$ and $F(p;z)$.
\end{remark}

\subsection{Perspectives}
In this paper, we have considered an elliptic analog of the Ding-Iohara algebra and a possibility of the free field realization of the elliptic Macdonald operator. In the following, we mention some ideas which can be cultivated in the future.

\subsubsection{Elliptic $q$-Virasoro algebra, elliptic $q$-$W_{N}$ algebra}
As we have shown, starting from the elliptic kernel function $\Pi(q,t,p)(x,y)$ we can construct the elliptic currents $\eta(p;z),\,\xi(p;z)$ and $\varphi^{\pm}(p;z)$ which satisfy the relations of the elliptic Ding-Iohara algebra. Furthermore we obtain the procedure of making elliptic currents, namely the method of elliptic deformation. Actually, we can apply the method of elliptic deformation to the free field realization of the $q$-Virasoro algebra, consequently an elliptic analog of the $q$-Virasoro algebra arises. Similarly, we can also construct free field realization of an elliptic analog of the $q$-$W_{N}$ algebra. In a near future, we would like to report these materials as the continuation of the present paper [26].

\subsubsection{Research of elliptic Macdonald symmetric functions}
To construct an elliptic analog of the Macdonald symmetric functions (in the following, we call it the elliptic Macdonald symmetric functions for short) is required for good understanding and research of some materials, for example the elliptic Ruijsenaars model [1], the superconformal index [22][23][24], etc. To construct the elliptic Macdonald symmetric functions, there would be a possibility to have an elliptic analog of the integral representations of the Macdonald symmetric functions. The integral representations of the Macdonald symmetric functions tells us that the Macdonald symmetric functions can be reproduced by the kernel function $\Pi(q,t)(x,y)$ and the weight function $\Delta(q,t)(x)$ defined by
\begin{align*}
\Delta(q,t)(x):=\prod_{i\neq{j}}\frac{(x_{i}/x_{j};q)_{\infty}}{(tx_{i}/x_{j};q)_{\infty}}
\end{align*}
and the ``seed" of the Macdonald symmetric functions [9]. The seed of the Macdonald symmetric functions are monomials. As is seen in the previous sections, we already have the elliptic kernel function $\Pi(q,t,p)(x,y)$, and the elliptic weight function $\Delta(q,t,p)(x)$ is also known [13] :
\begin{align*}
\Delta(q,t,p)(x):=\prod_{i\neq{j}}\frac{\Gamma_{q,p}(tx_{i}/x_{j})}{\Gamma_{q,p}(x_{i}/x_{j})}.
\end{align*}
But we don't know what is the seed of the elliptic Macdonald symmetric functions, i.e. the most simplest and nontrivial eigen functions of the elliptic Macdonald operator are not known. Therefore to construct an elliptic analog of the integral representations of the Macdonald symmetric functions is not accomplished.

On the other hand, it is known that the singular vectors of the $q$-Virasoro algebra and the $q$-$W_{N}$ algebra corresponds to the Macdonald symmetric functions [10][11][12]. Perhaps there would be a way to make the elliptic Macdonald symmetric functions from the elliptic analog of the $q$-Virasoro algebra. In the continuation paper [26], we construct an elliptic analog of the screening currents of the $q$-Virasoro algebra, and a correlation function of product of the elliptic screening currents reproduces the elliptic kernel function $\Pi(q,t,p)(x,y)$ as well as the elliptic weight function $\Delta(q,t,p)(x)$. But as we mentioned above, an elliptic analog of the integral representations of the Macdonald symmetric functions is not obtained yet.

\medskip
\textbf{Acknowledgement.}
The author would like to thank Koji Hasegawa and Gen Kuroki for helpful discussions and comments.

\section{Appendix}
\subsection{Appendix A : Boson calculus}
In this subsection we review some basic facts of boson calculus.

\begin{prop}
Let $\mathcal{A}$ be an associative $\mathbf{C}$-algebra. For $A\in\mathcal{A}$, we set the exponential of $A$ denoted by $e^{A}$ as follows.
\begin{align*}
e^{A}:=\exp(A):=\sum_{n\geq{0}}\frac{1}{n!}A^{n}.
\end{align*}
\textit{Then for $A,B\in{\mathcal{A}}$, the following holds :}
\begin{align*}
e^{A}Be^{-A}=e^{{\rm ad}(A)}B,
\end{align*}
where we define ${\rm ad}(A)B:=AB-BA$.
\end{prop}

\begin{proof}[\textit{Proof}]
Let us define $F(t):=e^{tA}Be^{-tA} \, (t\in\mathbf{C})$, then we can check the following :
\begin{align*}
\frac{d^{n}}{dt^{n}}F(t)\mid_{t=0}=\text{ad}(A)^{n}B \quad (n\geq0).
\end{align*}
By the Taylor expansion of $F(t)$ around $t=0$, we have 
\begin{align*}
F(t)=\sum_{n\geq 0}\frac{t^{n}}{n!}\frac{d^{n}}{dt^{n}}F(t)\mid_{t=0}=\sum_{n\geq 0}\frac{t^{n}}{n!}\text{ad}(A)^{n}B=e^{t\text{ad}(A)}B.
\end{align*}
From this expression of $F(t)$, we have $F(1)=e^{A}Be^{-A}=e^{\text{ad}(A)}B$. \quad $\Box$
\end{proof}

By this proposition, we have $e^{A}e^{B}e^{-A}=\exp(e^{\text{ad}(A)}B)$. Then the corollary holds.

\begin{cor}
For $A,B\in \mathcal{A}$, if $[A,B]\in{\mathbf{C}}$ we have the following :
\begin{align*}
e^{A}e^{B}=e^{[A,B]}e^{B}e^{A}.
\end{align*}
\end{cor}

This corollary is essentially the same as Wick's theorem which we use frequently in this paper.

Next we are going to prove Wick's theorem. First we set an associative $\mathbf{C}$-algebra denoted by $\mathcal{B}$ which is generated by $\{a_{n}\}_{n\in{\mathbf{Z}\setminus\{0\}}}$ and the defining relation :
\begin{align}
[a_{m},a_{n}]=A(m)\delta_{m+n,0} \quad (A(m)\in{\mathbf{C}}).
\end{align}
We call this type algebras \textit{bosons.} For example, if we choose $A(m)=m\displaystyle \frac{1-q^{|m|}}{1-t^{|m|}}$ then 
\begin{align}
[a_{m},a_{n}]=m\frac{1-q^{|m|}}{1-t^{|m|}}\delta_{m+n,0},
\end{align}
where this is one of the algebra of boson used in this paper. We define the normal ordering  $\bm{:} \bullet \bm{:}$ by
\begin{align}
\bm{:}a_{m}a_{n}\bm{:}=
\begin{cases}
a_{m}a_{n} \quad (m<n), \\
a_{n}a_{m} \quad (m\geq{n}).
\end{cases}
\end{align}
For $\{X_{n}\}_{n\in{\mathbf{Z}\setminus\{0\}}} \,(X_{n}\in{\mathbf{C}})$, we set $X(z)\in{\mathcal{B}\otimes\mathbf{C}[[z,z^{-1}]]}$ as a formal power series by
\begin{align*}
X(z):=\sum_{n\neq{0}}X_{n}a_{n}z^{-n}.
\end{align*}
We define \textit{the plus part of $X(z)$} denoted by $(X(z))_{+}$ and \textit{the minus part of $X(z)$} denoted by $(X(z))_{-}$ as follows :
\begin{align}
(X(z))_{+}:=\sum_{n>0}X_{n}a_{n}z^{-n}, \quad (X(z))_{-}:=\sum_{n<0}X_{n}a_{n}z^{-n}.
\end{align}
In this notation, we have 
\begin{align*}
\bm{:}\exp(X(z))\bm{:}=\exp((X(z))_{-})\exp((X(z))_{+}).
\end{align*}

\begin{prop}[\textbf{Wick's theorem}]
For boson operators $X(z)\in\mathcal{B}\otimes\mathbf{C}[[z,z^{-1}]]$ and $Y(w)\in\mathcal{B}\otimes\mathbf{C}[[w,w^{-1}]]$, if $[(X(z))_{+},(Y(w))_{-}]\in\mathbf{C}[[(w/z)]]$ exists we define $\langle X(z),Y(w) \rangle$ by
\begin{align}
\langle X(z),Y(w) \rangle:=[(X(z))_{+},(Y(w))_{-}].
\end{align}
Then we have
\begin{align}
\bm{:}\exp(X(z))\bm{:}\bm{:}\exp(Y(w))\bm{:}=\exp(\langle X(z),Y(w) \rangle)\bm{:}\exp(X(z))\exp(Y(w))\bm{:}.
\end{align}
\end{prop}

As an example of how to use Wick's theorem, we consider the boson algebra (6.2) and define $\eta(z)$ by
\begin{align*}
\eta(z):=\bm{:}\exp\bigg(-\sum_{n\neq{0}}(1-t^{n})a_{n}\frac{z^{-n}}{n}\bigg)\bm{:}.
\end{align*}
Let us show the following :
\begin{align*}
\eta(z)\eta(w)=\frac{(1-w/z)(1-qt^{-1}w/z)}{(1-qw/z)(1-t^{-1}w/z)}\bm{:}\eta(z)\eta(w)\bm{:}.
\end{align*}
By Wick's theorem, we have 
\begin{align*}
\eta(z)\eta(w)
&=\exp\bigg(\sum_{m>0}\sum_{n<0}(1-t^{m})(1-t^{n})[a_{m},a_{n}]\frac{z^{-m}w^{-n}}{mn}\bigg)\bm{:}\eta(z)\eta(w)\bm{:} \\
&=\exp\bigg(\sum_{m>0}\sum_{n<0}(1-t^{m})(1-t^{n})m\frac{1-q^{|m|}}{1-t^{|m|}}\delta_{m+n,0}\frac{z^{-m}w^{-n}}{mn}\bigg)\bm{:}\eta(z)\eta(w)\bm{:} \\
&=\exp\bigg(\sum_{m>0}(1-t^{m})(1-t^{-m})m\frac{1-q^{m}}{1-t^{m}}\frac{z^{-m}w^{m}}{m(-m)}\bigg)\bm{:}\eta(z)\eta(w)\bm{:} \\
&=\exp\bigg(-\sum_{m>0}(1-q^{m})(1-t^{-m})\frac{(w/z)^{m}}{m}\bigg)\bm{:}\eta(z)\eta(w)\bm{:} \\
&=\frac{(1-w/z)(1-qt^{-1}w/z)}{(1-qw/z)(1-t^{-1}w/z)}\bm{:}\eta(z)\eta(w)\bm{:},
\end{align*}
where we use $\log(1-x)=-\displaystyle \sum_{n>0}\frac{x^{n}}{n} \, (|x|<1)$.

\subsection{Appendix B : Some formulas}
In this subsection, we show some formulas which are used in this paper.

\subsubsection{Ramanujan's summation formula}
We show Ramanujan's summation formula which is used in section 4. As a preparation, we show the $q$-binomial theorem. In the following, we assume that the base $q\in{\mathbf{C}}$ satisfies $|q|<1$. Let us define
\begin{align*}
(x;q)_{\infty}:=\prod_{n\geq{0}}(1-xq^{n}), \quad (x;q)_{n}:=\frac{(x;q)_{\infty}}{(q^{n}x;q)_{\infty}} \quad (n\in{\mathbf{Z}}).
\end{align*}

\begin{prop}[\textbf{$q$-Binomial theorem}] 
For $a\in{\mathbf{C}}$, the following holds :
\begin{align*}
\frac{(az;q)_{\infty}}{(z;q)_{\infty}}=\sum_{n\geq{0}}\frac{(a;q)_{n}}{(q;q)_{n}}z^{n} \quad (|z|<1).
\end{align*}
\end{prop}

\begin{proof}[\textit{Proof}]
First, we expand $(az;q)_{\infty}/(z;q)_{\infty}$ as follows :
\begin{align*}
\frac{(az;q)_{\infty}}{(z;q)_{\infty}}=\sum_{n\geq 0}c_{n}z^{n}.
\end{align*}
Therefore what we have to show is $c_{n}=(a;q)_{n}/(q;q)_{n}$. Then by the equation
\begin{align*}
\frac{(aqz;q)_{\infty}}{(qz;q)_{\infty}}=\frac{1-z}{1-az}\frac{(az;q)_{\infty}}{(z;q)_{\infty}},
\end{align*}
we have the following :
\begin{align*}
q^{n}c_{n}-aq^{n-1}c_{n-1}=c_{n}-c_{n-1} \Longleftrightarrow c_{n}=\frac{1-aq^{n-1}}{1-q^{n}}c_{n-1}.
\end{align*}
By using this relation repeatedly, we have
\begin{align*}
c_{n}=\frac{1-aq^{n-1}}{1-q^{n}}\frac{1-aq^{n-2}}{1-q^{n-1}}\cdots\frac{1-a}{1-q}c_{0}=\frac{(a;q)_{n}}{(q;q)_{n}}c_{0}.
\end{align*}
Then $c_{0}=1$, hence we have $c_{n}=(a;q)_{n}/(q;q)_{n}$. \quad $\Box$
\end{proof}

In the $q$-binomial theorem, setting $a=0$ we have
\begin{align}
\frac{1}{(z;q)_{\infty}}=\sum_{n\geq 0}\frac{1}{(q;q)_{n}}z^{n}.
\end{align}
Similar to the proof of the $q$-binomial theorem, Euler's formula is shown :
\begin{align}
(z;q)_{\infty}=\sum_{n\geq 0}\frac{(-1)^{n}q^{\frac{n(n-1)}{2}}}{(q;q)_{n}}z^{n}.
\end{align}

Before giving the proof of Ramanujan's summation formula, we show Jacobi's triple product formula by using Euler's formula and the $q$-binomial theorem.

\begin{prop}[\textbf{Jacobi's triple product formula}]
\begin{align}
(q;q)_{\infty}(z;q)_{\infty}(qz^{-1};q)_{\infty}=\sum_{n\in{\mathbf{Z}}}(-1)^{n}z^{n}q^{\frac{n(n-1)}{2}}.
\end{align}
\end{prop}

\begin{proof}[\textit{Proof}]
First, we rewrite $(z;q)_{\infty}$ as 
\begin{align*}
(z;q)_{\infty}&=\sum_{n\geq{0}}\frac{(-1)^{n}q^{\frac{n(n-1)}{2}}}{(q;q)_{n}}z^{n}=\sum_{n\geq{0}}(-1)^{n}q^{\frac{n(n-1)}{2}}\frac{(q^{n+1};q)_{\infty}}{(q;q)_{\infty}}z^{n}\\
&=\frac{1}{(q;q)_{\infty}}\sum_{n\in{\mathbf{Z}}}(-1)^{n}q^{\frac{n(n-1)}{2}}(q^{n+1};q)_{\infty}z^{n} \quad (\because \text{For $n<0$, $(q^{n+1};q)_{\infty}=0$}).
\end{align*}
Furthermore, by applying Euler's formula to $(q^{n+1};q)_{\infty}$ we have
\begin{align*}
(z;q)_{\infty}&=\frac{1}{(q;q)_{\infty}}\sum_{n\in{\mathbf{Z}}}(-1)^{n}q^{\frac{n(n-1)}{2}}\sum_{r\geq{0}}\frac{(-1)^{r}q^{\frac{r(r-1)}{2}}}{(q;q)_{r}}(q^{n+1})^{r}z^{n}\\
&=\frac{1}{(q;q)_{\infty}}\sum_{\genfrac{}{}{0pt}{1}{n\in{\mathbf{Z}}}{r\geq{0}}}(-1)^{n+r}z^{n+r}q^{\frac{(n+r)(n+r-1)}{2}}\frac{1}{(q;q)_{r}}(qz^{-1})^{r} \\
&=\frac{1}{(q;q)_{\infty}}\sum_{r\geq{0}}\frac{1}{(q;q)_{r}}(qz^{-1})^{r}\sum_{n\in\mathbf{Z}}(-1)^{n+r}z^{n+r}q^{\frac{(n+r)(n+r-1)}{2}}.
\end{align*}
Then by substitution $n \to n-r$, we have
\begin{align*}
(z;q)_{\infty}=\frac{1}{(q;q)_{\infty}}\sum_{r\geq{0}}\frac{1}{(q;q)_{r}}(qz^{-1})^{r}\sum_{n\in\mathbf{Z}}(-1)^{n}z^{n}q^{\frac{n(n-1)}{2}}
=\frac{1}{(q;q)_{\infty}(qz^{-1};q)_{\infty}}\sum_{n\in{\mathbf{Z}}}(-1)^{n}z^{n}q^{\frac{n(n-1)}{2}}.
\end{align*}
Finally we have the following :
\begin{align*}
(q;q)_{\infty}(z;q)_{\infty}(qz^{-1};q)_{\infty}=\sum_{n\in{\mathbf{Z}}}(-1)^{n}z^{n}q^{\frac{n(n-1)}{2}}. \quad \Box
\end{align*}
\end{proof}

Jacobi's triple product formula means that
\begin{align*}
\Theta_{p}(z)=(p;p)_{\infty}(z;p)_{\infty}(pz^{-1};p)_{\infty}=\sum_{n\in{\mathbf{Z}}}(-1)^{n}z^{n}p^{\frac{n(n-1)}{2}}.
\end{align*}

Next let us prove Ramanujan's summation formula. We define the bilateral series $_{1}\psi_{1}\left(\genfrac{}{}{0pt}{0}{a}{b};z\right)$ by
\begin{align}
_{1}\psi_{1}\left(\genfrac{}{}{0pt}{0}{a}{b};z\right):=\sum_{n\in{\mathbf{Z}}}\frac{(a;q)_{n}}{(b;q)_{n}}z^{n}.
\end{align}
Then we can check the relation
\begin{align*}
(a;q)_{n}&=(1-a)(1-aq)\cdots(1-aq^{n-1}),\\
(a;q)_{-n}&=\frac{1}{(1-aq^{-n})(1-aq^{-n+1})\cdots(1-aq^{-1})}=\frac{q^{\frac{n(n+1)}{2}}}{(-1)^{n}a^{n}(q/a;q)_{n}} \quad (n>0).
\end{align*}
By these relations, the series $_{1}\psi_{1}\left(\genfrac{}{}{0pt}{0}{a}{b};z\right)$ can be rewritten as follows :
\begin{align*}
\sum_{n\in{\mathbf{Z}}}\frac{(a;q)_{n}}{(b;q)_{n}}z^{n}=\sum_{n\geq{0}}\frac{(a;q)_{n}}{(b;q)_{n}}z^{n}+\sum_{n\geq{1}}\frac{(q/b;q)_{n}}{(q/a;q)_{n}}\bigg(\frac{b}{az}\bigg)^{n}.
\end{align*}
Therefore the series $_{1}\psi_{1}\left(\genfrac{}{}{0pt}{0}{a}{b};z\right)$ converges in $|a^{-1}b|<|z|<1$.

\begin{prop}[\textbf{Ramanujan's summation formula}] 
For parameters $a,b\in{\mathbf{C}}$, the following holds :
\begin{align}
\sum_{n\in{\mathbf{Z}}}\frac{(a;q)_{n}}{(b;q)_{n}}z^{n}=\frac{(az;q)_{\infty}(q/az;q)_{\infty}(b/a;q)_{\infty}(q;q)_{\infty}}{(z;q)_{\infty}(b/az;q)_{\infty}(q/a;q)_{\infty}(b;q)_{\infty}} \quad (|a^{-1}b|<|z|<1).
\end{align}
\end{prop}

\begin{proof}[\textit{Proof}]
We are going to follow the proof due to Gasper and Rahman [4]. To prove this, we set 
\begin{align*}
f(b):={}_{1}\psi_{1}\left(\genfrac{}{}{0pt}{0}{a}{b};z\right)=\sum_{n\in{\mathbf{Z}}}\frac{(a;q)_{n}}{(b;q)_{n}}z^{n}
\end{align*}
as a function of $b$. We are going to show (6.11) by using a difference equation of $f(b)$. First, we have
\begin{align*}
_{1}\psi_{1}\left(\genfrac{}{}{0pt}{0}{a}{b};z\right)-a{}\,_{1}\psi_{1}\left(\genfrac{}{}{0pt}{0}{a}{b};qz\right)
&=\sum_{n\in{\mathbf{Z}}}\left\{\frac{(a;q)_{n}}{(b;q)_{n}}-a\frac{(a;q)_{n}}{(b;q)_{n}}q^{n}\right\}z^{n}
=\sum_{n\in{\mathbf{Z}}}\frac{(a;q)_{n+1}}{(b;q)_{n}}z^{n}\\
&=(1-b/q)\sum_{n\in{\mathbf{Z}}}\frac{(a;q)_{n+1}}{(b/q;q)_{n+1}}z^{n}
=(1-b/q)z^{-1}\sum_{n\in{\mathbf{Z}}}\frac{(a;q)_{n+1}}{(b/q;q)_{n+1}}z^{n+1}\\
&=(1-b/q)z^{-1}{}_{1}\psi_{1}\left(\genfrac{}{}{0pt}{0}{a}{b/q};z\right),
\end{align*}
therefore the following relation holds :
\begin{align*}
f(b)-(1-b/q)z^{-1}f(q^{-1}b)=a{}\,_{1}\psi_{1}\left(\genfrac{}{}{0pt}{0}{a}{b};qz\right).
\end{align*}
Taking a substitution $b \to qb$, we have the following :
\begin{align*}
f(qb)-(1-b)z^{-1}f(b)=a{}\,_{1}\psi_{1}\left(\genfrac{}{}{0pt}{0}{a}{qb};qz\right).
\end{align*}
Second, for $a{}\,_{1}\psi_{1}\left(\genfrac{}{}{0pt}{0}{a}{qb};qz\right)$, we have
\begin{align*}
a{}\,_{1}\psi_{1}\left(\genfrac{}{}{0pt}{0}{a}{qb};qz\right)
&=a\sum_{n\in{\mathbf{Z}}}\frac{(a;q)_{n}}{(qb;q)_{n}}q^{n}z^{n}\\
&=-ab^{-1}\sum_{n\in{\mathbf{Z}}}\frac{(a;q)_{n}(1-bq^{n}-1)}{(qb;q)_{n}}z^{n}\\
&=-ab^{-1}\sum_{n\in{\mathbf{Z}}}\frac{(a;q)_{n}(1-bq^{n})}{(qb;q)_{n}}z^{n}+ab^{-1}\sum_{n\in{\mathbf{Z}}}\frac{(a;q)_{n}}{(qb;q)_{n}}z^{n}\\
&=-ab^{-1}(1-b)\sum_{n\in{\mathbf{Z}}}\frac{(a;q)_{n}}{(b;q)_{n}}z^{n}+ab^{-1}f(qb)\\
&=-ab^{-1}(1-b)f(b)+ab^{-1}f(qb).
\end{align*}
Therefore we have 
\begin{align*}
f(qb)-(1-b)z^{-1}f(b)=-ab^{-1}(1-b)f(b)+ab^{-1}f(qb) 
\Longleftrightarrow \, f(b)=\frac{1-b/a}{(1-b)(1-b/az)}f(qb).
\end{align*}
By using this relation repeatedly, we have the following :
\begin{align*}
f(b)=\frac{(b/a;q)_{\infty}}{(b;q)_{\infty}(b/az;q)_{\infty}}f(0).
\end{align*}
Instead of $f(0)$, we determine $f(q)$. Then we have
\begin{align*}
f(q)&=\sum_{n\in{\mathbf{Z}}}\frac{(a;q)_{n}}{(q;q)_{n}}z^{n}=\sum_{n\geq{0}}\frac{(a;q)_{n}}{(q;q)_{n}}z^{n}\stackrel{\genfrac{}{}{0pt}{1}{\text{$q$-binomial}}{\text{theorem}}}{=}\frac{(az;q)_{\infty}}{(z;q)_{\infty}} \quad \left(\because \text{For $n<0$, \,$\frac{1}{(q;q)_{n}}=0$}\right) \\
&=\frac{(q/a;q)_{\infty}}{(q;q)_{\infty}(q/az;q)_{\infty}}f(0).
\end{align*}
Thus we have 
\begin{align*}
f(0)=\frac{(az;q)_{\infty}(q/az;q)_{\infty}(q;q)_{\infty}}{(z;q)_{\infty}(q/a;q)_{\infty}}.
\end{align*}
Consequently, we have the following :
\begin{align*}
f(b)=\sum_{n\in{\mathbf{Z}}}\frac{(a;q)_{n}}{(b;q)_{n}}z^{n}=\frac{(az;q)_{\infty}(q/az;q)_{\infty}(b/a;q)_{\infty}(q;q)_{\infty}}{(z;q)_{\infty}(b/az;q)_{\infty}(b;q)_{\infty}(q/a;q)_{\infty}}. \quad \Box
\end{align*}
\end{proof}

\subsubsection{Partial fraction expansion formula}
\begin{prop}[\textbf{Partial fraction expansion [17]}] 
Let $[u] \,(u\in\mathbf{C})$ be an entire function which satisfies the following relations :
\begin{align}
&(1) \, \text{Odd function} : [-u]=-[u], \\
&(2) \, \text{The Riemann relation} : \notag\\
&[x+z][x-z][y+w][y-w]-[x+w][x-w][y+z][y-z]=[x+y][x-y][z+w][z-w].
\end{align}
For $N\in\mathbf{Z}_{>0}$ and parameters $q_{i}, \,c_{i} \, (1\leq{i}\leq{N})$, set $c_{(N)}:=\sum_{i=1}^{N}c_{i}$. Then the following holds :
\begin{align}
\prod_{i=1}^{N}\frac{[u-q_{i}+c_{i}]}{[u-q_{i}]}=\sum_{i=1}^{N}\frac{[c_{i}]}{[c_{(N)}]}\frac{[u-q_{i}+c_{(N)}]}{[u-q_{i}]}\prod_{j\neq{i}}\frac{[q_{i}-q_{j}+c_{j}]}{[q_{i}-q_{j}]}.
\end{align}
\end{prop}

\begin{proof}[\textit{Proof}]
The proposition is shown by the Riemann relation and the induction of $N$. In the case $N=2$, we are going to show 
\begin{align*}
&\quad\frac{[u-q_{1}+c_{1}]}{[u-q_{1}]}\frac{[u-q_{2}+c_{2}]}{[u-q_{2}]} \\
&=\frac{[c_{1}]}{[c_{(2)}]}\frac{[u-q_{1}+c_{(2)}]}{[u-q_{1}]}\frac{[q_{1}-q_{2}+c_{2}]}{[q_{1}-q_{2}]}+\frac{[c_{2}]}{[c_{(2)}]}\frac{[u-q_{2}+c_{(2)}]}{[u-q_{2}]}\frac{[q_{2}-q_{1}+c_{1}]}{[q_{2}-q_{1}]}.
\end{align*}
By multiplying the both hand sides by $[c_{(2)}][u-q_{1}][u-q_{2}][q_{1}-q_{2}]$, the equation we have to show takes the form as follows :
\begin{align}
&\quad [c_{(2)}][q_{1}-q_{2}][u-q_{1}+c_{1}][u-q_{2}+c_{2}] \notag\\
&=[c_{1}][u-q_{1}+c_{(2)}][u-q_{2}][q_{1}-q_{2}+c_{2}]-[c_{2}][u-q_{2}+c_{(2)}][u-q_{1}][q_{2}-q_{1}+c_{1}].
\end{align}
Here we show (Right hand side of (6.15)) $\to$ (Left hand side of (6.15)) by the Riemann relation. Then we define $x,y,z,w$ by
\begin{align*}
x+y=c_{(2)}, \quad x-y=q_{1}-q_{2}, \quad z+w=u-q_{1}+c_{1}, \quad z-w=u-q_{2}+c_{2}.
\end{align*}
It is clear that $\text{(Left hand side of (6.15))}=[x+y][x-y][z+w][z-w]$. On the other hand, $x,y,z,w$ are
\begin{align*}
x=\frac{q_{1}-q_{2}+c_{(2)}}{2}, \quad y=\frac{q_{2}-q_{1}+c_{(2)}}{2}, \quad z=\frac{2u-q_{1}-q_{2}+c_{(2)}}{2}, \quad w=\frac{-q_{1}+q_{2}+c_{1}-c_{2}}{2},
\end{align*}
hence we have 
\begin{align*}
\text{(Right hand side of (6.15))}
&\quad=[x+w][y+z][z-y][x-w]-[y-w][x+z][z-x][y+w] \\
&\quad=[x+z][x-z][y+w][y-w]-[x+w][x-w][y+z][y-z] \\
&\stackrel{\genfrac{}{}{0pt}{1}{\text{Riemann}}{\text{relation}}}{=}[x+y][x-y][z+w][z-w],
\end{align*}
therefore (6.15) is satisfied.

Next we suppose that the relation (6.14) holds for $\exists N\geq 2$. Then in the case of $N+1$, we have the following :
\begin{align*}
&\quad \prod_{i=1}^{N+1}\frac{[u-q_{i}+c_{i}]}{[u-q_{i}]} \\
&=\frac{[u-q_{N+1}+c_{N+1}]}{[u-q_{N+1}]}\prod_{i=1}^{N}\frac{[u-q_{i}+c_{i}]}{[u-q_{i}]} \\
&=\frac{[u-q_{N+1}+c_{N+1}]}{[u-q_{N+1}]}\sum_{i=1}^{N}\frac{[c_{i}]}{[c_{(N)}]}\frac{[u-q_{i}+c_{(N)}]}{[u-q_{i}]}\prod_{1\leq{j}\leq{N},\,j\neq{i}}\frac{[q_{i}-q_{j}+c_{j}]}{[q_{i}-q_{j}]} \\
&=\sum_{i=1}^{N}\frac{[c_{i}]}{[c_{(N)}]}\bigg\{\frac{[c_{N+1}]}{[c_{(N+1)}]}\frac{[u-q_{N+1}+c_{(N+1)}]}{[u-q_{N+1}]}\frac{[q_{N+1}-q_{i}+c_{(N)}]}{[q_{N+1}-q_{i}]} \\
&\hskip 3cm +\frac{[c_{(N)}]}{[c_{(N+1)}]}\frac{[u-q_{i}+c_{(N+1)}]}{[u-q_{i}]}\frac{[q_{i}-q_{N+1}+c_{N+1}]}{[q_{i}-q_{N+1}]}\bigg\}\prod_{1\leq{j}\leq{N},\,j\neq{i}}\frac{[q_{i}-q_{j}+c_{j}]}{[q_{i}-q_{j}]} \\
&=\frac{[c_{N+1}]}{[c_{(N+1)}]}\frac{[u-q_{N+1}+c_{(N+1)}]}{[u-q_{N+1}]}\sum_{i=1}^{N}\frac{[c_{i}]}{[c_{(N)}]}\frac{[q_{N+1}-q_{i}+c_{(N)}]}{[q_{N+1}-q_{i}]}\prod_{1\leq{j}\leq{N},\,j\neq{i}}\frac{[q_{i}-q_{j}+c_{j}]}{[q_{i}-q_{j}]} \\
&\hskip 3cm +\sum_{i=1}^{N}\frac{[c_{i}]}{[c_{(N+1)}]}\frac{[u-q_{i}+c_{(N+1)}]}{[u-q_{i}]}\prod_{1\leq{j}\leq{N+1},\,j\neq{i}}\frac{[q_{i}-q_{j}+c_{j}]}{[q_{i}-q_{j}]}.
\end{align*} 
From the hypothesis of the induction, we have
\begin{align*}
\sum_{i=1}^{N}\frac{[c_{i}]}{[c_{(N)}]}\frac{[q_{N+1}-q_{i}+c_{(N)}]}{[q_{N+1}-q_{i}]}\prod_{1\leq{j}\leq{N},\,j\neq{i}}\frac{[q_{i}-q_{j}+c_{j}]}{[q_{i}-q_{j}]}
=\prod_{j=1}^{N}\frac{[q_{N+1}-q_{j}+c_{j}]}{[q_{N+1}-q_{j}]}.
\end{align*}
Therefore we have 
\begin{align*}
\prod_{i=1}^{N+1}\frac{[u-q_{i}+c_{i}]}{[u-q_{i}]}
&=\frac{[c_{N+1}]}{[c_{(N+1)}]}\frac{[u-q_{N+1}+c_{(N+1)}]}{[u-q_{N+1}]}\prod_{j=1}^{N}\frac{[q_{N+1}-q_{j}+c_{j}]}{[q_{N+1}-q_{j}]} \\
&\qquad+\sum_{i=1}^{N}\frac{[c_{i}]}{[c_{(N+1)}]}\frac{[u-q_{i}+c_{(N+1)}]}{[u-q_{i}]}\prod_{1\leq{j}\leq{N+1},\,j\neq{i}}\frac{[q_{i}-q_{j}+c_{j}]}{[q_{i}-q_{j}]} \\
&=\sum_{i=1}^{N+1}\frac{[c_{i}]}{[c_{(N)}]}\frac{[u-q_{i}+c_{(N)}]}{[u-q_{i}]}\prod_{1\leq{j}\leq{N+1},\,j\neq{i}}\frac{[q_{i}-q_{j}+c_{j}]}{[q_{i}-q_{j}]},
\end{align*}
then this shows that the relation (6.14) holds in the case of $N+1$. By the induction the proof of the proposition 6.7 is complete. \quad $\Box$
\end{proof}

The proposition 6.7 is written by additive variables. Let us rewrite it by multiplicative variables. The theta function is defined by
\begin{align*}
\Theta_{p}(x)=(p;p)_{\infty}(x;p)_{\infty}(px^{-1};p)_{\infty}.
\end{align*}

\begin{prop}[\textbf{The Riemann relation of the theta function $\Theta_{p}(x)$}] 
For the theta function $\Theta_{p}(x)$, the Riemann relation holds as follows :
\begin{align}
&\Theta_{p}(xz)\Theta_{p}(x/z)\Theta_{p}(yw)\Theta_{p}(y/w)-\Theta_{p}(xw)\Theta_{p}(x/w)\Theta_{p}(yz)\Theta_{p}(y/z) \notag\\
&\hskip 7cm =\frac{y}{z}\Theta_{p}(xy)\Theta_{p}(x/y)\Theta_{p}(zw)\Theta_{p}(z/w).
\end{align}
\end{prop}

\begin{proof}[\textit{Sketch of the proof}]
For $x\in{\mathbf{C}}$, we set $f(x)$ as the ratio of the right hand side of (6.16) and the left hand side of (6.16) :
\begin{align}
f(x):=\frac{(y/z)\Theta_{p}(xy)\Theta_{p}(x/y)\Theta_{p}(zw)\Theta_{p}(z/w)}{\Theta_{p}(xz)\Theta_{p}(x/z)\Theta_{p}(yw)\Theta_{p}(y/w)-\Theta_{p}(xw)\Theta_{p}(x/w)\Theta_{p}(yz)\Theta_{p}(y/z)}. 
\end{align}
Then we can check $f(px)=f(x)$ by the property of the theta function $\Theta_{p}(px)=-x^{-1}\Theta_{p}(x)$. Besides $f(x)$ has no poles in the region $|p|\leq|x|\leq1$. This shows that $f(x)$ is bounded on $\mathbf{C}^{\times}$. By the Liouville theorem $f(x)$ is constant, i.e. the ratio (6.17) does not depends on $x$. Since the ratio (6.17) does not depends on $x$, we have $f(x)=f(w)=1$. Therefore (6.16) holds. \quad $\Box$
\end{proof}

For a variable $z\in{\mathbf{C}}$, we define the additive variable $u\in{\mathbf{C}}$ as $z=e^{2\pi iu}$. In this notation, we set $[u]$ as
\begin{align}
[u]:=-z^{-1/2}\Theta_{p}(z).
\end{align}
Using this notation, the Riemann relation of the theta function (6.16) is rewritten to the form (6.13). Consequently, we have 

\begin{prop}
Set $z,\,x_{j},\,t_{j} \,(1\leq{j}\leq{N})$ by
\begin{align*}
z:=e^{2\pi iu}, \quad x_{j}:=e^{2\pi iq_{j}}, \quad t_{j}:=e^{2\pi ic_{j}}, \quad t_{(N)}:=t_{1}t_{2}\cdots t_{N},
\end{align*}
where $u,\,q_{j},\,c_{j}\,(1\leq{j}\leq{N})$ are variable and parameters in the proposition 6.7. From (6.14), a product of the theta function is decomposed into the partial fraction as follows :
\begin{align}
\prod_{i=1}^{N}\frac{\Theta_{p}(t_{i}x_{i}^{-1}z)}{\Theta_{p}(x_{i}^{-1}z)}=\sum_{i=1}^{N}\frac{\Theta_{p}(t_{i})}{\Theta_{p}(t_{(N)})}\frac{\Theta_{p}(t_{(N)}x_{i}^{-1}z)}{\Theta_{p}(x_{i}^{-1}z)}\prod_{j\neq{i}}\frac{\Theta_{p}(t_{j}x_{i}/x_{j})}{\Theta_{p}(x_{i}/x_{j})}.
\end{align}
\end{prop}

The theta function $\Theta_{p}(x)$ satisfies $\Theta_{p}(x) \xrightarrow[p \to 0]{} 1-x$. From the trigonometric limit of (6.19), hence we have
\begin{align}
\prod_{i=1}^{N}\frac{1-t_{i}x_{i}^{-1}z}{1-x_{i}^{-1}z}=\sum_{i=1}^{N}\frac{1-t_{i}}{1-t_{(N)}}\frac{1-t_{(N)}x_{i}^{-1}z}{1-x_{i}^{-1}z}\prod_{j\neq{i}}\frac{1-t_{j}x_{i}/x_{j}}{1-x_{i}/x_{j}}.
\end{align}

\end{document}